\documentclass[12pt, reqno]{amsart}           
\usepackage{amsmath, amsthm, amscd, amsfonts, amssymb, graphicx, color}
\usepackage[bookmarksnumbered, colorlinks, plainpages]{hyperref}
\usepackage{xcolor}
\usepackage{xcolor}
\usepackage{ulem}

\hypersetup{colorlinks=true,linkcolor=red, anchorcolor=green, citecolor=cyan, urlcolor=red, filecolor=magenta, pdftoolbar=true}
\input{mathrsfs.sty}
\newtheorem{theorem}{Theorem}[section]
\newtheorem{lemma}[theorem]{Lemma}
\newtheorem{prop}[theorem]{Proposition}
\newtheorem{cor}[theorem]{Corollary}
\theoremstyle{definition}

\theoremstyle{remark}
\newtheorem{remark}[theorem]{\bf{Remark}}
\numberwithin{equation}{section}

\allowdisplaybreaks
\begin{document}
\title[]
{\small{ $A$-Davis-Wielandt Radius Bounds of Semi-Hilbertian Space Operators}}

	\author[]
 {Messaoud Guesba, Somdatta Barik, Pintu Bhunia, Kallol Paul}
	
	\address[Guesba]{Faculty of Exact Sciences, Department of Mathematics, El Oued
		University, 39000 Algeria}
	\email{guesbamessaoud2@gmail.com}

\address[Barik] {Department of Mathematics, Jadavpur University, Kolkata 700032, West Bengal, India}
\email{bariksomdatta97@gmail.com}
\address[Bhunia]{Department of Mathematics, Indian Institute of Science, Bengaluru 560012, Karnataka, India}
\email{pintubhunia5206@gmail.com}

\address[Paul] {Department of Mathematics, Jadavpur University, Kolkata 700032, West Bengal, India}
\email{kalloldada@gmail.com}

\thanks{ The authors would like to acknowledge the referee for his/her valuable comments towards improve the article. Miss Somdatta Barik would like to thank UGC, Govt. of India for the financial support in the form of JRF under the mentorship of Prof. Kallol Paul.  Dr. Pintu Bhunia would like to thank SERB, Govt. of India for the financial support in the form of National Post Doctoral Fellowship (N-PDF, File No. PDF/2022/000325) under the mentorship of Prof. Apoorva Khare}

\maketitle

\begin{abstract}
Consider $\mathcal{H}$ is a complex Hilbert space and $A$ is a positive operator on $\mathcal{H}.$
 The mapping $\langle\cdot,\cdot\rangle_A: \mathcal{H}\times \mathcal{H} \to \mathbb {C}$, defined as $\left\langle y,z\right\rangle
_{A}=\left\langle  Ay,z\right\rangle $ for all $y,z$ $\in $ ${\mathcal{H}}$,
induces a seminorm $ \left\Vert \cdot\right\Vert _{A}$. The $A$-Davis-Wielandt radius of an
operator $S$ on $\mathcal{H}$ is defined as
$d\omega _{A}\left( S\right) =\sup \left\{  \sqrt{\left\vert \left\langle
Sz,z\right\rangle _{A}\right\vert ^{2}+\left\Vert Sz\right\Vert _{A}^{4}}
:\left\Vert z\right\Vert _{A}=1\right\} \text{.}
$ 
We investigate some new bounds for $d\omega _{A}\left( S\right)$ which refine the existing bounds. We also give some bounds
for the $2\times 2$ off-diagonal block 
matrices.

\ \ \ \ \ \ \ \ \ \ \ \ \ \ \ \ \ \ \ \ \ \ \ \ \ \ \ \ \ \ \ \ \ \ \ \ \ \ 
\ \ \ \ \ \ \ \ \ \ \ \ \ \ \ \ \ \ \ \ \ \ \ \ \ \ \ \ \ \ \ \ \ \ \ \ \ 

\ \ \ \ \ \ \ \ \ \ \ \ \ \ \ \ \ \ \ \ \ \ \ \ \ \ \ \ \ \ \ \ \ \ \ \ \ \ 
\ \ \ \ \ \ \ \ \ \ \ \ \ \ \ \ \ \ \ \ \ \ \ \ \ \ \ \ \ \ \ \ \ \ \ \ \ \ 
\ \ \ \ \ \ \ \ \ \ \ \ \ \ \ \ \ \ \ \ \ \ \ \ \ \ \ \ \ \ \ \ \ \ \ \ \ \ 
\ \ \ \ \ \ \ \ \ \ \ \ \ \ \ \ \ \ \ \ \ \ \ \ \ \ \ \ \ \ \ \ \ \ \ \ \ \ 
\ \ \ \ \ \ \ \ \ \ \ \ \ \ \ \ \ \ \ \ \ \ \ \ \ \ 
\end{abstract}

\section{\textbf{Introduction and preliminaries}{\protect\Large \ }}

\noindent Here we collect all the technical ingredients that are
necessary to read the paper. 
Throughout, ${\mathcal{H}}$ is a complex Hilbert space and  ${\mathcal{B}}({\mathcal{H}})$ is the $\mathbb{C}^*$-algebra
of bounded operators on ${\mathcal{H}}$. For
 $S\in $ ${\mathcal{B}}({\mathcal{H}})$, the null space, range space
and closure of the range space of $S$ are respectively denoted as ${\mathcal{N}}\left(
S\right) $, ${\ \mathcal{R}}\left( S\right) $ and $\overline{{\mathcal{R}}
\left( S\right) }.$  
A positive operator $A\in {\mathcal{B}}({\mathcal{H}})^{}$ defines a sesquilinear form 
$
\left\langle \cdot ,\cdot \right\rangle _{A}:{\mathcal{H}}\times {\mathcal{H}
}\rightarrow 
\mathbb{C}
\text{ defined as }\left\langle x,z\right\rangle _{A} =\left\langle
Ax,z\right\rangle \text{.}
$
This gives a seminorm $\left\Vert
\cdot\right\Vert _{A}$ defined by
$
\left\Vert z\right\Vert _{A}=\left\Vert A^{\frac{1}{2}}z\right\Vert \text{.}
$
Clearly, $\left\Vert z\right\Vert _{A}=0$ iff $z\in {\mathcal{
 N}}\left( A\right) $. Therefore, $\left\Vert \cdot\right\Vert _{A}$ is a norm
on ${\ \mathcal{H}}$ iff $A$ is one-one.
Suppose
$\left\Vert S\right\Vert _{A}=\sup\limits_{\substack{ z\in \overline{{\ 
\mathcal{R}}\left( A\right) }  \\ z\neq 0}}\frac{\left\Vert Sz\right\Vert
_{A}}{\left\Vert z\right\Vert _{A}}=\sup\limits_{\substack{ z\in \overline{{%
\ \ \mathcal{R}}\left( A\right) }  \\ \left\Vert z\right\Vert _{A}=1}}
\left\Vert Sz\right\Vert _{A} \text{}$ is the $A$-operator seminorm of $S\in {\mathcal{B}}({\mathcal{H}})$. 
If there is a scalar $c>0$ such that $\left\Vert Sz\right\Vert _{A}\leq c\left\Vert
z\right\Vert _{A}$ for all $z\in \overline{{\mathcal{R}}\left( A\right) }$,
then $\left\Vert S\right\Vert _{A}<\infty $.
If $\|S\|_A<\infty$, then
$
\left\Vert S\right\Vert _{A}=\sup \left\{ \left\vert \left\langle
Sx,z\right\rangle _{A}\right\vert \text{: }x,z\in \overline{{\mathcal{R}}
\left( A\right) }, \, \left\Vert x\right\Vert _{A}=\left\Vert
z\right\Vert _{A}=1\text{ }\right\} \text{.}
$
An operator $R\in {\mathcal{B}}({ 
\mathcal{H}})$ is $A$-adjoint of $S\in {\mathcal{B}}({\mathcal{H}})$  if 
$
\left\langle Sx,z\right\rangle _{A}=\left\langle x,Rz\right\rangle _{A}\text{
for all }x,z\in {\mathcal{H}}\text{,}
$
or equivalently, 
$
AR=S^{\ast }A\text{.}
$
 If $AS=S^{\ast }A$ then $S$ is $A$-selfadjoint and if $AS$ is positive then $S$ is
 $A$-positive and we write $S\geq _{A}0$.
 Consider  ${\mathcal{B}}_{A}({\  
\mathcal{H}})$ is the collection of
operators that admit $A$-adjoints.
From Douglas result \cite{DO}, we have
$
{\mathcal{B}}_{A}({\mathcal{H}}) =\left\{ S\in {\mathcal{B}}({\mathcal{H}}
):{\mathcal{R}}\left( S^{\ast }A\right) \subseteq {\mathcal{R}}\left( A\right)
\right\} 
= \left\{ S\in {\mathcal{B}}({\mathcal{H}}):\text{ $\exists$}~ c>0:\left\Vert
ASz\right\Vert \leq c\left\Vert Az\right\Vert ,\forall \text{ }z\in {\ 
\mathcal{H}}\right\} \text{.}
$
For $S\in {\mathcal{B}}_{A}({\mathcal{H}})$, $S$ admits $A$-adjoint. Furthermore, there exists a distinguished $A$-adjoint of $S$, namely, the reduced solution of  $AZ=S^{\ast }A$, i.e., $
S^{\sharp _{A}}=A^{\dag }S^{\ast }A$, where $A^{\dag }$ is the Moore-Penrose inverse of $A$. The operator $S^{\sharp _{A}}$ satisfies 
$
AS^{\sharp _{A}}=S^{\ast }A\text{, }{\mathcal{R}}\left( S^{\sharp
_{A}}\right) \subseteq \overline{{\mathcal{R}}\left( A\right) }\text{, }{\ 
\mathcal{N}}\left( S^{\sharp _{A}}\right) ={\mathcal{N}}\left( S^{\ast
}A\right) \text{.}
$
Again, by Douglas result, we can see that
$
{\mathcal{B}}_{A^{{1}/{2}}}({\mathcal{H}}) = \{ S\in {\mathcal{B}}({ 
\mathcal{H}}):\text{$\exists$}~ c>0:\left\Vert Sz\right\Vert _{A}\leq c\left\Vert
z\right\Vert _{A},\forall \text{ }z\in {\mathcal{H}} \} \text{.}
$
If $S\in 
{\mathcal{B}}_{A^{{1}/{2}}}({\mathcal{H}})$, 
 $
\left\Vert S \right\Vert _{A} =\sup\limits_{z\notin {\mathcal{N}}\left(
A\right) }\frac{\left\Vert Sz\right\Vert _{A}}{\left\Vert z\right\Vert _{A}}
=\sup\limits_{\substack{ z\in {\mathcal{H}}  \\ \left\Vert z\right\Vert
_{A}=1 }}\left\Vert Sz\right\Vert _{A}\text{,}
$ see \cite{AR2}.
If $S \in {\mathcal{B}}_{A^{{1}/{2}}}({\mathcal{H}})$ then $S\left( {\mathcal{N}}\left(
A\right) \right) \subseteq $ ${\mathcal{N}}\left( A\right) $ and 
 $
\left\Vert Sz\right\Vert _{A}\leq \left\Vert S\right\Vert _{A}\left\Vert
z\right\Vert _{A}\text{ for all} \text{ }z\in {\mathcal{H}}\text{.}
$
The sets ${\mathcal{B}}_{A}({\mathcal{H}})$ and ${\mathcal{B}}_{A^{{1}/{
2}}}({\mathcal{H}})$ are subalgebras
 (see 
\cite{AR2}) and 
$
{\mathcal{B}}_{A}({\mathcal{H}})\subseteq {\mathcal{B}}_{A^{{1}/{2}}}({\ 
\mathcal{H}}).$
Note that $S$ is $A$-selfadjoint, but $
S=S^{\sharp _{A}}$ may not be true. For example, consider $A=\left( 
\begin{array}{cc}
1 & 1 \\ 
1 & 1%
\end{array}%
\right) \in {\mathcal{M}}_{2}$\ $\left( 
\mathbb{C}
\right) ^{+}$\ \ and $S=$\ $\left( 
\begin{array}{cc}
2 & 2 \\ 
0 & 0%
\end{array}%
\right) $. Here $S$ is $A$-selfadjoint, however $S^{\sharp _{A}}=$ $\left( 
\begin{array}{cc}
1 & 1 \\ 
1 & 1%
\end{array}%
\right) \neq $ $S$. Note that, $S=S^{\sharp _{A}}$ iff  ${\mathcal{R}}\left( S^{\sharp _{A}}\right) \subseteq $ ${
 \mathcal{R}}\left( A\right) $ and $S$  is $
A $-selfadjoint.  An operator $S$ in ${\mathcal{B}}_{A}({%
\mathcal{H}})$ is $A$-unitary if $\|Sx\|_A=\|S^{\sharp_A}x\|_A=\|x\|_A,$ $
\forall x\in\mathcal{H}.$
For more we refer to \cite{AR1, AR2}. If $S\in {\mathcal{B}}_{A}({\mathcal{H
}})$ then the following holds:

(1) If $AS=SA$, then $S^{\sharp _{A}}=P_{\overline{\mathcal{R }\left(
A\right) } }S^{\ast },$ where $P_{\overline{\mathcal{R }\left( A\right) } }$
denote the orthogonal projection
on $\overline{\mathcal{R }\left( A\right) }.$

(2) $S^{\sharp _{A}}S$, $SS^{\sharp _{A}}$ are $A$-selfadjoint and are $A$-positive.

(3) $S^{\sharp _{A}}\in {\mathcal{B}}_{A}({\mathcal{H}})$, $\left( S^{\sharp
_{A}}\right) ^{\sharp _{A}}=P_{\overline{\mathcal{R }\left( A\right) }}SP_{ 
\overline{\mathcal{R }\left( A\right) }}$, $\left( \left( S^{\sharp
_{A}}\right) ^{\sharp _{A}}\right) ^{^{\sharp _{A}}}=S^{\sharp _{A}}$.

(4) If $R\in {\mathcal{B}}_{A}({\mathcal{H}})$ then $SR\in {\mathcal{B}}_{A}(%
{\mathcal{H}})$ and $\left( SR\right) ^{\sharp _{A}}=R^{\sharp_{A}}S^{^{\sharp _{A}}}$. Furthermore, for all $S, R\in {\mathcal{B}}_{A^{{1}/{2}}}({\mathcal{H}})$, we have
\begin{equation}
\left\Vert SR\right\Vert _{A}\leq \left\Vert S\right\Vert _{A}\left\Vert
R\right\Vert _{A}\text{.}  \label{2}
\end{equation}

(5) $\left\Vert S^{\sharp _{A}}\right\Vert
_{A}= \left\Vert S\right\Vert _{A}= \left\Vert S^{\sharp _{A}}S\right\Vert _{A}^{\frac{1}{2}}=\left\Vert
SS^{\sharp _{A}}\right\Vert _{A}^{\frac{1}{2}}$.

To simplify we write $P$ instead of $P_{\overline{
\mathcal{R }\left( A\right) }}$ . An operator  $R$ in $ {\mathcal{B}}_{A}({\mathcal{H}}%
)$ can be written as $R=\Re_A (R)+i\Im_A(R),$ where $\Re_A (R)=\frac{%
R+R^{\sharp_{A}}}{2}$ and $\Im_A (R)=\frac{R-R^{\sharp_{A}}}{2i}.$ 
Recall that, 
$\omega \left( S\right) :=\sup  \{ |\left\langle Sz,z\right\rangle| :z\in {
 \mathcal{H}}, \left\Vert z\right\Vert =1 \} \text{}$ is the numerical radius of $S\in {\mathcal{B}}({%
\mathcal{H}})$,
 see \cite{GUS}.
For more facts, see \cite{D3,GUE4,KI3,AL} and references therein.
One of the 
generalization of the numerical radius is $A$-numerical radius of  $S\in {\mathcal{B}%
}_{A}({\mathcal{H}})$, it is defined by
$
\omega _{A}\left( S\right) :=\sup \left\{ \left\vert \left\langle
Sz,z\right\rangle _{A}\right\vert :z\in {\mathcal{H}}\text{, }\left\Vert
z\right\Vert _{A}=1\right\} \text{.}
$
Obviously, for $A=I$ we obtain the (classical) numerical radius which received considerable attention in recent years by
many researchers. We refer to \cite{BH,BH2,GUE3,GUE1} and
references therein.
Clearly,
$
\omega _{A}\left( S\right) =\omega _{A}\left( S^{\sharp _{A}}\right) \text{
for any }S\in {\mathcal{B}}_{A}({\mathcal{H}})\text{.}
$
Also, the $A$-numerical radius satisfies the power
inequality (see \cite{MOS, ZAM2}): for $S\in {\mathcal{B}}%
_{A^{{1}/{2}}}({\mathcal{H}})$, 
$
\omega _{A}\left( S^{n}\right) \leq \omega _{A}^{n}\left( S\right) \text{, \ 
}n=1,2,...\text{.}
$
It defines a seminorm on ${\mathcal{B}}_{A^{1/2}}(%
{\mathcal{H}})$, and  it satisfies 
\begin{equation}
\frac{1}{2}\left\Vert S \right\Vert _{A}\leq \omega _{A}\left( S\right) \leq
\left\Vert S\right\Vert _{A} \quad \text{for every $S\in {\mathcal{B}}_{A^{{1}/{2}}}
({\mathcal{H}}) $}\text{.}  \label{00}
\end{equation}
Here 
$
\omega _{A}\left( S \right) =\left\Vert S \right\Vert _{A}\text{}
$ if $S$ is $A$-selfadjoint.  
Recall that, the $A$-Crawford number of $S\in {\mathcal{B}}_{A^{1/2}}({\mathcal{H}})$ is
defined by 
$
c_{A}\left( S\right) =\inf \left\{ \left\vert \left\langle Sz,z\right\rangle
_{A}\right\vert :z\in {\mathcal{H}}\text{,}\left\Vert z\right\Vert
_{A}=1\right\} \text{.}
$
For more basic facts, we refer
 to \cite{GUE5,Sadi, ZAM2}.
The Davis-Wielandt radius of $S\in {
 \mathcal{B}}({\mathcal{H}})\mathcal{\ }$ is defined as
$
d\omega \left( S\right) =\sup \left\{ \left( {\left\vert \left\langle
Sz,z\right\rangle \right\vert ^{2}+\left\Vert Tz\right\Vert ^{4}}\right)^{1/2}: z\in {\ 
\mathcal{H}}\text{, }\left\Vert z\right\Vert =1\right\} \text{.}
$
For more about Davis-Wielandt radius we refer to \cite{BH02,BH4,BH3,Li,ZA3,ZA4}.
The $A$-Davis-Wielandt radius of
 $S\in {\mathcal{B}}_{A^{{1}/{2}}}({\mathcal{H}})$ is
defined (see \cite{BHA}) as
$
d\omega _{A}\left( S\right) =\sup \left\{ \left({\left\vert \left\langle
Sz,z\right\rangle _{A}\right\vert ^{2}+\left\Vert Sz\right\Vert _{A}^{4}}\right)^{1/2}
:z\in {\mathcal{H}}\text{, }\left\Vert z\right\Vert _{A}=1\right\} \text{.}
$
It satisfies 
$
\max\left\{ \omega _{A}\left( S\right) ,\left\Vert S\right\Vert
_{A}^{2}\right\} \leq d\omega _{A}\left( S\right) \leq \sqrt{\omega
_{A}^{2}\left( S\right) +\left\Vert S\right\Vert _{A}^{4}}\text{.}
$

Very recently some results on the $A$-Davis-Wielandt radius are studied (see \cite{BHA}). In this manuscript we continue the study in
this direction and we provide some new bounds of the $A$-Davis-Wielandt radius. Further, we give some new bounds for the Davis-Wielandt radius of  $2\times 2$ off-diagonal block matrices.

\section{Bounds for the $A$-Davis-Wielandt radius of operators}

To develop our first bound of $A$-Davis-Wielandt radius, the following inner-product inequality is necessary.

  
  \begin{lemma}
  	\cite{BSBP1,KZB} \label{KZ} If $a,b,c\in {\mathcal{H}}$ and $
  	\left\Vert c\right\Vert _{A}=1$ and $0\neq \alpha\in \mathbb {C}
   $, then  
  	\begin{equation*}
  		\left\vert \left\langle a,c\right\rangle _{A}\left\langle c,b\right\rangle
  		_{A}\right\vert \leq \frac{1}{|\alpha|}\Big( \max \{1, |\alpha-1|\}\left\Vert a\right\Vert
  		_{A}\left\Vert b\right\Vert _{A}+\left\vert \left\langle a,b\right\rangle
  		_{A}\right\vert \Big) \text{.}
  	\end{equation*}
  	In particular, for $\alpha =2,$ 
  	$$   \left\vert \left\langle a,c\right\rangle _{A}\left\langle c,b\right\rangle
  	_{A}\right\vert \leq \frac{1}{2}\Big( \left\Vert a\right\Vert
  	_{A}\left\Vert b\right\Vert _{A}+\left\vert \left\langle a,b\right\rangle
  	_{A}\right\vert \Big) \text{.}$$
  \end{lemma}

\begin{theorem}\label{theo1}
Let $S\in {\mathcal{B}}_{A}({\mathcal{H}})$ and $0\neq \alpha\in \mathbb {C}$. Then  \begin{equation*}
d\omega _{A}^{2}\left( S \right) \leq \omega _{A}^{2}\left( S^{\sharp _{A}}S+S\right) +\frac{2}{|\alpha|}\omega _{A}\left(
S^{\sharp _{A}}S^{2}\right) +\frac{\max \{1, |\alpha-1|\}}{|\alpha|}\left\Vert \left( S^{\sharp_{A}}S\right) ^{2}+S^{\sharp _{A}}S\right\Vert _{A}\text{.}
\end{equation*}
\end{theorem}

\begin{proof}
Take $z\in {\mathcal{H}}$ and $\|z\|_A=1$. We obtain  
\begin{eqnarray*}
&&\left\vert \left\langle Sz,z\right\rangle _{A}\right\vert ^{2}+\left\Vert
Sz\right\Vert _{A}^{4} \\
&=&\left\vert \left\langle Sz,z\right\rangle _{A}+\left\langle
Sz,Sz\right\rangle _{A}\right\vert ^{2}-2\Re_A\left( \left\langle
Sz,Sz\right\rangle _{A}\left\langle Sz,z\right\rangle _{A}\right) \\
&\leq &\left\vert \left\langle \left( S+S^{\sharp _{A}}S\right) z,z\right\rangle
_{A}\right\vert ^{2}+2\left\vert \left\langle S^{\sharp
_{A}}Sz,z\right\rangle _{A}\left\langle z,Sz\right\rangle _{A}\right\vert \\
&\leq &\left\vert \left\langle \left( S+S^{\sharp _{A}}S\right)
z,z\right\rangle _{A}\right\vert ^{2}+\frac{2}{|\alpha|} \Big[\max \{1, |\alpha-1|\}\left\Vert S^{\sharp
_{A}}Sz\right\Vert _{A}\left\Vert Sz\right\Vert _{A}+\left\vert \left\langle
S^{\sharp _{A}}Sz,Sz\right\rangle _{A}\right\vert\Big] \\
&&\text{(by Lemma \ref{KZ})} \\
&\leq &\left\vert \left\langle \left( S+S^{\sharp _{A}}S\right)
z,z\right\rangle _{A}\right\vert ^{2}+\frac{\max \{1, |\alpha-1|\}}{|\alpha|} \left( \left\Vert S^{\sharp
	_{A}}Sz\right\Vert _{A}^{2}+\left\Vert Sz\right\Vert _{A}^{2}\right)+\frac{2}{|\alpha|}\left\vert \left\langle
S^{\sharp _{A}}Sz,Sz\right\rangle _{A}\right\vert \\
&&\text{(by the AM-GM inequality)} \\
&=&\left\vert \left\langle \left( S+S^{\sharp _{A}}S\right) z,z\right\rangle
_{A}\right\vert ^{2}+\frac{\max \{1, |\alpha-1|\}}{|\alpha|}\left( \left\langle S^{\sharp
_{A}}Sz,S^{\sharp _{A}}Sz\right\rangle _{A}+\left\langle S^{\sharp
_{A}}Sz,z\right\rangle _{A}\right) \\
&&+\frac{2}{|\alpha|}\left\vert \left\langle z,S^{\sharp
_{A}}S^{2}z\right\rangle _{A}\right\vert \\
&=&\left\vert \left\langle \left( S+S^{\sharp _{A}}S\right) z,z\right\rangle
_{A}\right\vert ^{2}+\frac{\max \{1, |\alpha-1|\}}{|\alpha|}\left( \left\langle \left( S^{\sharp
_{A}}S\right) ^{2}z,z\right\rangle _{A}+\left\langle S^{\sharp
_{A}}Sz,z\right\rangle _{A}\right) \\
&&+\frac{2}{|\alpha|}\left\vert \left\langle z,S^{\sharp
_{A}}S^{2}z\right\rangle _{A}\right\vert \\
&\leq &\omega _{A}^{2}\left( S^{\sharp _{A}}S+S\right) +\frac{2}{|\alpha|}\omega _{A}\left(
S^{\sharp _{A}}S^{2}\right) +\frac{\max \{1, |\alpha-1|\}}{|\alpha|}\left\Vert \left( S^{\sharp
_{A}}S\right) ^{2}+S^{\sharp _{A}}S\right\Vert _{A}\text{.}
\end{eqnarray*}
This gives the bound as desired. 
\end{proof}

Now from Theorem \ref{theo1} (for $\alpha=2$), we deduce the bound:
\begin{cor}\label{KZ1}
	If $S\in {\mathcal{B}}_{A}({\mathcal{H}})$ then
	\begin{equation*}
		d\omega _{A}^{2}\left( S\right) \leq \omega _{A}^{2}\left( S^{\sharp
			_{A}}S+S\right) +\omega _{A}\left( S^{\sharp _{A}}S^{2}\right) +\frac{1}{2}
		\left\Vert \left( S^{\sharp _{A}}S\right) ^{2}+S^{\sharp _{A}}S\right\Vert
		_{A}\text{.}
	\end{equation*}
\end{cor}

Again, considering $\alpha=n$ and $n\to \infty$ in Theorem \ref{theo1}, we get the bound:
\begin{cor}
	If $S\in {\mathcal{B}}_{A}({\mathcal{H}})$ then
	\begin{equation*}
		d\omega _{A}^{2}\left( S\right) \leq \omega _{A}^{2}\left( S^{\sharp
			_{A}}S+S\right) +
		\left\Vert \left( S^{\sharp _{A}}S\right) ^{2}+S^{\sharp _{A}}S\right\Vert
		_{A}\text{.}
	\end{equation*}
\end{cor}

To get next result we need the following semi-inner product inequality.

\begin{lemma}
	\cite{KZLAA} \label{KZLAA1} If $a,b,c\in {\mathcal{H}}$ and $
	\left\Vert c\right\Vert _{A}=1$ then 
	\begin{eqnarray*}
		\left\vert \left\langle a,c\right\rangle _{A}\left\langle c,b\right\rangle
		_{A}\right\vert \leq \frac{1}{2}\left( \left\Vert a\right\Vert
		_{A}\left\Vert b\right\Vert _{A}+\left\vert \left\langle a,b\right\rangle
		_{A}\right\vert \right)-\delta(a, b, c),
	\end{eqnarray*}
		\textit{where} 
	\resizebox{.9\hsize}{!}{$\delta(a, b,c)=	\begin{cases}
			\frac{\|b\|_A}{\|a\|_A}\left( |\langle a,c\rangle_A|\underset{\lambda\in \mathbb C}{\inf}\|c-\lambda b\|_A-\frac12\underset{\mu\in \mathbb C}{\inf}\|a-\mu b\|_A\right)^2& \textit{if}\,\, \|a\|_A\|b\|_A\neq0\\
			0 & \textit{if}\,\,  \|a\|_A\|b\|_A =0.
		\end{cases}$}
\end{lemma}

\begin{theorem}
	If $S\in {\mathcal{B}}_{A}({\mathcal{H}})$ then  
	\begin{equation*}
	d\omega _{A}^{2}\left( S\right) \leq
	 \omega _{A}^{2}\left( S^{\sharp _{A}}S+S\right) +\omega _{A}\left(
		S^{\sharp _{A}}S^{2}\right) +\frac{1}{2}\left\Vert \left( S^{\sharp
			_{A}}S\right) ^{2}+S^{\sharp _{A}}S\right\Vert _{A}-2\underset{\|z\|_A=1}{\inf}\delta(S^{\sharp _{A}}Sz, Sz, z),
	\end{equation*}
	\textit{where}
 \resizebox{.9\hsize}{!}
{$\delta(a, b,c)=	\begin{cases}
		\frac{\|b\|_A}{\|a\|_A}\left( |\langle a,c\rangle_A|\underset{\lambda\in \mathbb C}{\inf}\|c-\lambda b\|_A-\frac12\underset{\mu\in \mathbb C}{\inf}\|a-\mu b\|_A\right)^2& \textit{if}\,\, \|a\|_A\|b\|_A\neq0\\
		0 & \textit{if}\,\, \|a\|_A\|b\|_A=0.
	\end{cases}$}
\end{theorem}
\begin{proof}
	Take $z\in {\mathcal{H}}$ and $\|z\|_A=1$. Using Lemma \ref{KZLAA1}, we obtain
	\begin{eqnarray*}
		&&\left\vert \left\langle Sz,z\right\rangle _{A}\right\vert ^{2}+\left\Vert
		Sz\right\Vert _{A}^{4} \\
		&=&\left\vert \left\langle Sz,z\right\rangle _{A}+\left\langle
		Sz,Sz\right\rangle _{A}\right\vert ^{2}-2\Re_A\left( \left\langle
		Sz,Sz\right\rangle _{A}\left\langle Sz,z\right\rangle _{A}\right) \\
		&\leq &\left\vert \left\langle \left( S+S^{\sharp _{A}}S\right) z,z\right\rangle
		_{A}\right\vert ^{2}+2\left\vert \left\langle S^{\sharp
			_{A}}Sz,z\right\rangle _{A}\left\langle z,Sz\right\rangle _{A}\right\vert\\
		&\leq& 	\left\vert \left\langle \left( S+S^{\sharp _{A}}S\right) z,z\right\rangle
		_{A}\right\vert ^{2}+\left\Vert S^{\sharp
			_{A}}Sz\right\Vert _{A}\left\Vert Sz\right\Vert _{A}+\left\vert \left\langle
		S^{\sharp _{A}}Sz,Sz\right\rangle _{A}\right\vert
		-2\delta(S^{\sharp _{A}}Sz, Sz, z) \\
		&\leq &\left\vert \left\langle \left( S+S^{\sharp _{A}}S\right) z,z\right\rangle
		_{A}\right\vert ^{2}+\frac{1}{2}\left( \left\langle \left( S^{\sharp
			_{A}}S\right) ^{2}z,z\right\rangle _{A}+\left\langle S^{\sharp
			_{A}}Sz,z\right\rangle _{A}\right)\\
			&& +
		\left\vert \left\langle z,S^{\sharp
			_{A}}S^{2}z\right\rangle _{A}\right\vert-2\delta(T^{\sharp _{A}}Sz, Sz, z)\\
			&\leq& \omega _{A}^{2}\left(S+ S^{\sharp
				_{A}}S\right) +\omega _{A}\left( S^{\sharp _{A}}S^{2}\right) +\frac{1}{2}
			\left\Vert \left( S^{\sharp _{A}}S\right) ^{2}+S^{\sharp _{A}}S\right\Vert
			_{A}-2\delta(S^{\sharp _{A}}Sz, Sz, z),
			\end{eqnarray*}
which implies the desired bound.
\end{proof}

We now obtain both upper and lower bounds. 

\begin{theorem}
\label{th3} If $S\in \mathcal{B}_{A}(\mathcal{H}),$ then  
\begin{eqnarray*}
&&\max \left\{ \omega_{A}(\Re _{A}(S)+iS^{\sharp _{A}}S),\omega_{A}(\Im
_{A}(S)+iS^{\sharp _{A}}S)\right\} \leq d\omega_{A}(S) \\
&\leq &\min \left\{ \sqrt{\omega_{A}^{2}(\Re _{A}(S)+iS^{\sharp
_{A}}S)+\Vert \Im _{A}(S)\Vert _{A}^{2}},\sqrt{\omega_{A}^{2}(\Im
_{A}(S)+iS^{\sharp _{A}}S)+\Vert \Re _{A}(S)\Vert _{A}^{2}}\right\} .
\end{eqnarray*}
\end{theorem}

\begin{proof}
Take $z\in {\mathcal{H}}$ and $\|z\|_A=1$. We have  
\begin{eqnarray}  \label{lb3}
\left|\langle Sz,z \rangle_A\right|^2 +\|Sz\|^4_A&&=\left|\langle
(\Re_A(S)+i\Im_A(S))z,z \rangle_A\right|^2 +\langle S^{\sharp_A}Sz,z
\rangle_A^2  \notag \\
&& = \left|\langle \Re_A(S)z,z \rangle_A\right|^2+\left|\langle \Im_A(S)z,z
\rangle_A\right|^2+\langle S^{\sharp_A}Sz,z \rangle_A^2  \notag \\
&& =|\langle (\Re_A(S)+iS^{\sharp_A}S)z,z\rangle_A|^2+\left|\langle
\Im_A(S)z,z \rangle_A\right|^2 \\
&&\leq \omega^2_A(\Re_A(S)+iS^{\sharp_A}S)+\|\Im_A(S)\|^2_A .  \notag
\end{eqnarray}
This implies 
\begin{align}  \label{lb1}
d\omega^2_A(S)\leq \omega^2_A(\Re_A(S)+iS^{\sharp_A}S)+\|\Im_A(S)\|^2_A.
\end{align}
Similarly, we obtain  
\begin{align}  \label{lb2}
d\omega^2_A(S)\leq \omega^2_A(\Im_A(S)+iS^{\sharp_A}S)+\|\Re_A(S)\|^2_A.
\end{align}
Now combining the inequalities \eqref{lb1} and \eqref{lb2}, we get the
desired upper bound.
Again from \eqref{lb3}, we get  
\begin{align}
\left|\langle Sz,z \rangle_A\right|^2 +\|Sz\|^4_A\geq |\langle
(\Re_A(S)+iS^{\sharp_A}S)z,z\rangle_A|^2
\end{align}
and  
\begin{align}
\left|\langle Sz,z \rangle_A\right|^2 +\|Sz\|^4_A\geq |\langle
(\Im_A(S)+iS^{\sharp_A}S)z,z\rangle_A|^2.
\end{align}
Now first taking the supremum over $\|z\|_A=1$ and then
combining them we obtain  the desired lower bound.
\end{proof}

From Theorem \ref{th3}, we now obtain the result which gives the
equality for the $A$-numerical radius.

\begin{cor}
If $S\in \mathcal{B}_{A}(\mathcal{H})$ with $(\Re _{A}(S^{\sharp
_{A}}))^{2}=\Im _{A}(S^{\sharp _{A}})$ then 
\begin{equation*}
\omega _{A}(S)=\Vert \Re _{A}(S)\Vert _{A}\sqrt{1+\Vert \Re _{A}(S)\Vert
_{A}^{2}}.
\end{equation*}
\end{cor}

\begin{proof}
First we suppose $S\in \mathcal{B}_{A}(\mathcal{H})$ with $S=S^{\sharp _{A}}.$
Then $S$ is $A$-selfadjoint and so $d\omega _{A}(S)=\sqrt{\Vert S\Vert
_{A}^{2}+\Vert S\Vert _{A}^{4}}$ (see \cite{BHA}). From Theorem \ref{th3},
it follows $d\omega _{A}(S)=\omega _{A}(S+iS^{2}).$ Hence, $\omega
_{A}(S+iS^{2})=\sqrt{\Vert S\Vert _{A}^{2}+\Vert S\Vert _{A}^{4}}.$ Now substitute $
S$ by $\Re _{A}(S^{\sharp _{A}})$ for $S\in \mathcal{B}_{A}(\mathcal{H})$, the result derives from the facts $\omega _{A}(S)=\omega
_{A}(S^{\sharp _{A}})$ and $\Vert \Re _{A}(S)\Vert _{A}=\Vert \Re
_{A}(S^{\sharp _{A}})\Vert _{A}.$ 
\end{proof}

\begin{remark}
Theorem \ref{th3} is better than the existing bound in
\cite{BHA}, namely  
\begin{align}  \label{eq3}
d\omega^2_A(S)\geq 2\max\{\omega_A(S) c_A(S^{\sharp_A}S), c_A(S) \|S\|_A^2\}
\end{align}
for some operators.  
If we assume $A= 
\begin{pmatrix}
1 & 0 \\ 
0 & 2%
\end{pmatrix}
$ and $S= 
\begin{pmatrix}
1 & 0 \\ 
0 & 0%
\end{pmatrix}
$ then \eqref{eq3} gives $d\omega_A(S)\geq0$, but Theorem \ref%
{th3} gives $d\omega_A(S)\geq \sqrt{2}.$ 

\end{remark}

A generalization of the Cauchy-Schwarz inequality is as the lemma.

\begin{lemma}
\cite{L22}  \label{LL} If $S\in {\mathcal{B}}_{A}({\mathcal{H}})$ and $
x,z\in {\  \mathcal{H}}$ with $\left\Vert
z\right\Vert _{A}=\left\Vert x\right\Vert  _{A}=1$ then  
\begin{equation*}
\left\vert \left\langle Sx,z\right\rangle _{A}\right\vert ^{2}\leq
\left\langle S^{\sharp _{A}}Sx,x\right\rangle _{A}^{\frac{1}{2}}\left\langle
SS^{\sharp _{A}}z,z\right\rangle _{A}^{\frac{1}{2}}\text{.}
\end{equation*}
\end{lemma}



We now obtain an upper bound. 

\begin{theorem}
\label{th11}  If $S\in {\mathcal{B}}_{A}({\mathcal{H}})$ then  
\begin{equation*}
d\omega _{A}^{2}\left( S\right) \leq \left\Vert S\right\Vert
_{A}^{4}+2\left\Vert S \right\Vert _{A}^{2}-\sqrt{c_{A}\left( S^{\sharp
_{A}}S\right) c_{A}\left( SS^{\sharp _{A}}\right) }\text{.}
\end{equation*}
\end{theorem}

\begin{proof}
Take $z\in {\mathcal{H}}$ and $\|z\|_A=1$. We obtain
\begin{eqnarray*}
&&\left\vert \left\langle Sz,z\right\rangle _{A}\right\vert ^{2}+\left\Vert
Sz\right\Vert _{A}^{4} \\
&\leq &\left\langle S^{\sharp _{A}}Sz,z\right\rangle _{A}^{\frac{1}{2}%
}\left\langle SS^{\sharp _{A}}z,z\right\rangle _{A}^{\frac{1}{2}}+\left\Vert
Sz\right\Vert _{A}^{4} \\
&&\text{(by Lemma \ref{LL})} \\
&\leq &\frac{1}{2}\left( \left\langle S^{\sharp _{A}}Sz,z\right\rangle
_{A}+\left\langle SS^{\sharp _{A}}z,z\right\rangle _{A}\right) +\left\Vert
Sz\right\Vert _{A}^{4} \\
&&\text{(by the AM-GM inequality)} \\
&=&\frac{1}{2}\left( \left\langle S^{\sharp _{A}}Sz,z\right\rangle _{A}^{%
\frac{1}{2}}+\left\langle SS^{\sharp _{A}}z,z\right\rangle _{A}^{\frac{1}{2}%
}\right) ^{2}-\left\langle S^{\sharp _{A}}Sz,z\right\rangle _{A}^{\frac{1}{2}%
}\left\langle SS^{\sharp _{A}}z,z\right\rangle _{A}^{\frac{1}{2}}+\left\Vert
Sz\right\Vert _{A}^{4} \\
&\leq &\frac{1}{2}\left( \left\Vert S^{\sharp _{A}}Sz\right\Vert _{A}^{\frac{%
1}{2}}+\left\Vert SS^{\sharp _{A}}z\right\Vert _{A}^{\frac{1}{2}}\right)
^{2}+\left\Vert Sz\right\Vert _{A}^{4}-\left\langle T^{\sharp
_{A}}Sz,z\right\rangle _{A}^{\frac{1}{2}}\left\langle SS^{\sharp
_{A}}z,z\right\rangle _{A}^{\frac{1}{2}} \\
&\leq &\frac{1}{2}\left( \left\Vert S^{\sharp _{A}}S\right\Vert _{A}^{\frac{1%
}{2}}+\left\Vert SS^{\sharp _{A}}\right\Vert _{A}^{\frac{1}{2}}\right)
^{2}+\left\Vert S\right\Vert _{A}^{4}-\sqrt{c_{A}\left( S^{\sharp
_{A}}S\right) c_{A}\left( SS^{\sharp _{A}}\right) } \\
&=&2\left\Vert S\right\Vert _{A}^{2}+\left\Vert S\right\Vert _{A}^{4}-\sqrt{%
c_{A}\left( S^{\sharp _{A}}S\right) c_{A}\left( SS^{\sharp _{A}}\right) }%
\text{,}
\end{eqnarray*}
which gives the desired bound.  
\end{proof}

\begin{remark}
\label{re11}  Theorem \ref{th11} is better than the existing
bound in \cite{KF_HJM}, namely  
\begin{align}  \label{eq11}
d\omega _{A}^{2}\left( S\right)\leq \max\{\|S\|_A^2, \|S\|_A^4 \}+\omega
_{A}\left(S^{\sharp_A}S^2 \right)
\end{align}
for some operators.  For example, if we consider $A=S=
\begin{pmatrix}
1 & 0 \\ 
0 & 2%
\end{pmatrix}%
$ then \eqref{eq11} gives $d\omega _{A}^{2}\left( S\right)\leq 24$, whereas
Theorem \ref{th11} gives $d\omega _{A}^{2}\left( S\right)\leq 23.$
\end{remark}

We next obtain a bound which is a refinement of the existing
bound given in \cite{KF_HJM}.

\begin{theorem}
\label{th8} If $S\in\mathcal{B}_A(\mathcal{H})$ then  
\begin{align*}
d\omega_A^2(S)\leq
\omega_A^2(e^{i\theta}S+S^{\sharp_A}S)+2\|S\|^2_A\|\Re_A(e^{i\theta}S)\|_A,%
\,\, \forall~\theta\in\mathbb{R}.
\end{align*}
\end{theorem}

\begin{proof}
For any $z\in {\mathcal{H}}$ and $\|z\|_A=1$, we obtain 
\begin{eqnarray*}
\left|\langle Sz,z \rangle_A\right|^2 +\|Sz\|^4_A&&=\left|\langle
(S+S^{\sharp_A}S)z,z \rangle_A \right|^2-2\|Sz\|^2_A\langle \Re_A(S)z,z
\rangle_A \\
&&\leq \left|\langle (S+S^{\sharp_A}S)z,z \rangle_A
\right|^2+2\|Sz\|^2_A\langle \Re_A(S)z,z \rangle_A \\
&&\leq \omega_A^2(S+S^{\sharp_A}S)+2\|S\|^2_A\|\Re_A(S)\|_A.
\end{eqnarray*}
This implies 
\begin{align}
d\omega_A^2(S)\leq \omega_A^2(S+S^{\sharp_A}S)+2\|S\|^2_A\|\Re_A(S)\|_A.
\end{align}
Replacing $S$ by $e^{i\theta}S,$ we obtain the desired bound.
\end{proof}

\begin{remark}
Following Theorem \ref{th8}, we get  
\begin{eqnarray}  \label{p02}
d\omega^2_A(S)\leq \omega_A^2(S \pm S^{\sharp_A}S)+2\Vert S \Vert_A^2 \Vert
\Re_A(S) \Vert_A.
\end{eqnarray}
Clearly, \eqref{p02} is sharper than the existing bound in \cite{KF_HJM} ,
namely  
\begin{eqnarray}  \label{p03}
d\omega^2_A(S)\leq \omega_A^2(S^{\sharp_A}S-S)+2\Vert S \Vert_A^2 \Vert
\omega_A(S) \Vert_A.
\end{eqnarray}
\end{remark}

We now obtain a lower bound.

\begin{theorem}
\label{th10} If $S\in\mathcal{B}_A(\mathcal{H})$ then  
\begin{equation*}
d\omega_A^2(S)+2\|S\|^2_A\|\Re_A(S)\|_A\geq \max\left\{
\omega_A^2(S+S^{\sharp_A}S), \omega_A^2(S-S^{\sharp_A}T)\right\}.
\end{equation*}
\end{theorem}

\begin{proof}
Take $z\in {\mathcal{H}}$, $\|z\|_A=1$, we get  
\begin{eqnarray*}
\left|\langle Sz,z \rangle_A\right|^2 +\|Sz\|^4_A
&& \geq \left|\langle (S+S^{\sharp_A}S)z,z \rangle_A
\right|^2-2\|Sz\|^2_A|\langle \Re_A(S)z,z \rangle_A| \\
&&\geq \left|\langle (S+S^{\sharp_A}S)z,z \rangle_A
\right|^2-2\|S\|^2_A\|\Re_A(S)\|_A.
\end{eqnarray*}
From this we get  
\begin{align}  \label{al1}
d\omega_A^2(S)+2\|S\|^2_A\|\Re_A(S)\|_A\geq \omega_A^2(S+S^{\sharp_A}S).
\end{align}
Replacing $S$ by $-S,$ we get  
\begin{align}  \label{al2}
d\omega_A^2(S)+2\|S\|^2_A\|\Re_A(S)\|_A\geq \omega_A^2(S-S^{\sharp_A}S).
\end{align}
Combining \eqref{al1} and \eqref{al2}, we obtain the bound as desired.
\end{proof}

\begin{remark}
\label{re12}  Theorem \ref{th10} is better than the existing
bound in \cite{BHA}, namely  
\begin{eqnarray}  \label{eq10}
d\omega_A^2(S)\geq 2\max \{\omega_A(S) c_A(S^{\sharp_A}S), c_A(S)\|S\|_A^2\}
\end{eqnarray}
for some operators. If we consider the same example as in Remark \ref{re11}
then  \eqref{eq10} gives $d\omega_A^2(S)\geq 8$ but Theorem \ref{th10} gives 
$ d\omega_A^2(S)\geq 20.$
\end{remark}

Now the following lemmas are needed.

\begin{lemma}
\cite{C} \label{L1.2}\bigskip\ If $S\in {\mathcal{B}}({\mathcal{H}})$ is 
$A$-positive, and  $z\in {\mathcal{H}}$, $\left\Vert 
z\right\Vert _{A}=1$ then  
\begin{equation*}
\left\langle Sz,z\right\rangle _{A}^{n}\leq \left\langle
S^{n}z,z\right\rangle _{A}\text{ for all }n\in 
\mathbb{N}
^{ }\text{.}
\end{equation*}
\end{lemma}

 \begin{lemma}\label{KKK}\cite{KZLAA}
	If $x, z\in {\mathcal{H}}$ and $\|x\|_A\neq0$ then
	$$ |\langle x,z \rangle_A|\leq    \|z\|_A \left( \|x\|_A-\frac{\underset{\lambda\in\mathbb C}{\inf}\|x-\lambda z\|^2_A}{2\|x\|_A}\right) .$$
\end{lemma}

We now develop a bound.

\begin{theorem}
\label{th17} Suppose $S\in {\mathcal{B}}_{A}({\mathcal{H}})$. Then for all $\alpha\in[0,1],$ 
\begin{equation*}
d\omega _{A}^{2}\left( S\right) \leq \left\Vert \alpha S^{\sharp
	_{A}}S+\left( 1-\alpha \right) SS^{\sharp _{A}}+\left( S^{\sharp
	_{A}}S\right) ^{2}\right\Vert _{A}-\left(\alpha \mu+(1-\alpha)\eta\right),
\end{equation*}
\textit{where} 
\begin{equation*}
	\mu=\underset{\|z\|_A=1}{\inf}\left[\underset{\lambda\in\mathbb C}{\inf}\|Sz-\lambda z\|^2_A-\frac{\underset{\lambda\in\mathbb C}{\inf}\|Sz-\lambda z\|^4_A}{4\|Sz\|^2_A}\right]
\end{equation*}
\textit{and}
 \begin{equation*}
	\eta=\underset{\|z\|_A=1}{\inf}\left[\underset{\lambda\in\mathbb C}{\inf}\|S^{\sharp_{A}}z-\lambda z\|^2_A-\frac{\underset{\lambda\in\mathbb C}{\inf}\|S^{\sharp_{A}}z-\lambda z\|^4_A}{4\|S^{\sharp_{A}}z\|^2_A}\right].
\end{equation*} 
\end{theorem}

\begin{proof}
Take $z\in {\mathcal{H}}$, $\|z\|_A=1$ and  $\alpha \in \left[ 
0,1 \right] $. By using lemma \ref{KKK}, we obtain
\begin{eqnarray*}
\left\vert \left\langle Sz,z\right\rangle _{A}\right\vert ^{2} 
&=&\alpha \left\vert \left\langle Sz,z\right\rangle _{A}\right\vert
^{2}+\left( 1-\alpha \right) \left\vert \left\langle z,S^{\sharp
_{A}}z\right\rangle _{A}\right\vert ^{2} \\
&\leq &\alpha\left( \|Sz\|_A-\frac{\underset{\lambda\in\mathbb C}{\inf}\|Sz-\lambda z\|^2_A}{2\|Sz\|_A}\right)^2\\
&&+\left( 1-\alpha \right)
\left( \|S^{\sharp
	_{A}}z\|_A-\frac{\underset{\lambda\in\mathbb C}{\inf}\|S^{\sharp
		_{A}}z-\lambda z\|^2_A}{2\|S^{\sharp
		_{A}}z\|_A}\right)^2 \\
&=&\alpha \|Sz\|^2_A+ (1-\alpha)\|S^{\sharp
	_{A}}z\|^2_A-\alpha \left(\underset{\lambda\in\mathbb C}{\inf}\|Sz-\lambda z\|^2_A-\frac{\underset{\lambda\in\mathbb C}{\inf}\|Sz-\lambda z\|^4_A}{4\|Sz\|^2_A}\right)\\
	&&-(1-\alpha)\left(\underset{\lambda\in\mathbb C}{\inf}\|S^{\sharp_{A}}z-\lambda z\|^2_A-\frac{\underset{\lambda\in\mathbb C}{\inf}\|S^{\sharp_{A}}z-\lambda z\|^4_A}{4\|S^{\sharp_{A}}z\|^2_A}\right)\\
&\leq&\left\langle \left( \alpha S^{\sharp _{A}}S+\left( 1-\alpha \right)
SS^{\sharp _{A}}\right) z,z\right\rangle _{A}-\left(\alpha \mu+(1-\alpha)\eta\right)\text{.}
\end{eqnarray*}
Therefore, we have  
\begin{eqnarray*}
&&\left\vert \left\langle Sz,z\right\rangle _{A}\right\vert ^{2}+\left\Vert
Sz\right\Vert _{A}^{4} \\
&=&\left\langle \left( \alpha S^{\sharp _{A}}S+\left( 1-\alpha \right)
SS^{\sharp _{A}}\right) z,z\right\rangle _{A}+\left\langle S^{\sharp
_{A}}Sz,z\right\rangle _{A}^{2}-\left(\alpha \mu+(1-\alpha)\eta\right) \\
&\leq &\left\langle \left( \alpha S^{\sharp _{A}}S+\left( 1-\alpha \right)
SS^{\sharp _{A}}\right) z,z\right\rangle _{A}+\left\langle \left( S^{\sharp
_{A}}S\right) ^{2}z,z\right\rangle _{A}-\left(\alpha \mu+(1-\alpha)\eta\right) \\
&&\text{(by Lemma \ref{L1.2})} \\
&\leq &\left\Vert \alpha S^{\sharp _{A}}S+\left( 1-\alpha \right) SS^{\sharp
_{A}}+\left( S^{\sharp _{A}}S\right) ^{2}\right\Vert _{A}-\left(\alpha \mu+(1-\alpha)\eta\right)\text{.}
\end{eqnarray*}
This gives 
\begin{equation}
d\omega _{A}^{2}\left( S \right) \leq \left\Vert \alpha S^{\sharp
_{A}}S+\left( 1-\alpha \right) SS^{\sharp _{A}}+\left( S^{\sharp
_{A}}S\right) ^{2}\right\Vert _{A}-\left(\alpha \mu+(1-\alpha)\eta\right)\text{.}  \label{C}
\end{equation}

\end{proof}

\begin{remark}
\label{re181} (i) Form Theorem \ref{th17} (for $\alpha=0$) we obtain 
\begin{align}  \label{eq17}
d\omega _{A}^{2}\left( S \right) \leq \|SS^{\sharp _{A}}+(S^{\sharp
_{A}}S)^2\|_A-\eta.
\end{align}
(ii) The bound \eqref{eq17} is better than the existing bound in \cite%
{KF_HJM}, namely 
\begin{align}  \label{eq18}
d\omega _{A}^{2}\left( S \right) \leq \sqrt{\omega_{A}\left((S^{\sharp
_{A}}S)^2+(S^{\sharp _{A}}S)^4\right)+2 \omega_{A}^2(S^{\sharp _{A}}S^2)}
\end{align}
for some operators. For example, if we consider $A=%
\begin{pmatrix}
1 & 0 \\ 
0 & 2%
\end{pmatrix}%
$ and $S=
\begin{pmatrix}
0 & 1 \\ 
0 & 0%
\end{pmatrix}%
$ then \eqref{eq18} gives $d\omega _{A}^{2}\left( S\right) \leq \frac{\sqrt 5%
}{4}$ but \eqref{eq17} gives $d\omega _{A}^{2}\left( S\right) \leq\frac{1}{2%
}.$
\end{remark}
   

We now present another upper bound.

\begin{theorem}
\label{THH1}\bigskip If $S\in {\mathcal{B}}_{A}({\mathcal{H}})$ then  
$
d\omega _{A}^{2}\left( S\right) \leq \min \left\{ \beta ,\gamma \right\} 
\text{,}
$
where 
\begin{equation*}
\beta =\min_{\alpha \in \left[ 0,1\right] }\left\{ \left\Vert \frac{\alpha }{
4}S^{\sharp _{A}}S+\left( 1-\frac{3 \alpha }{4} \right) SS^{\sharp
_{A}}+\left( S^{\sharp _{A}}S\right) ^{2}\right\Vert _{A}+\frac{\alpha }{2}
\omega _{A}\left( S^2\right)-(1-\alpha)\eta \right\},
\end{equation*}
\begin{equation*}
\gamma =\min_{\alpha \in \left[ 0,1\right] }\left\{ \left\Vert \left( 1- 
\frac{3 \alpha}{4} \right) S^{\sharp _{A}}S+\frac{\alpha }{4}SS^{\sharp
_{A}}+\left( S^{\sharp _{A}}S\right) ^{2}\right\Vert _{A}+\frac{\alpha }{2}
\omega _{A}\left( S^2\right)-(1-\alpha)\mu \right\}
\end{equation*}
and $\mu, \eta$ are as in Theorem \ref{th17}.
  \end{theorem}
  
\begin{proof}
Suppose $z\in {\mathcal{H}}$, $\|z\|_A=1$ and $\alpha \in \left[ 
0,1 \right] $. Then we obtain 
\begin{eqnarray*}
\left\vert \left\langle Sz,z\right\rangle _{A}\right\vert ^{2} 
&=&\alpha \left\vert \left\langle Sz,z\right\rangle _{A}\right\vert
^{2}+\left( 1-\alpha \right) \left\vert \left\langle z,T^{\sharp
_{A}}z\right\rangle _{A}\right\vert ^{2} \\
&\leq &\alpha \left\vert \left\langle Sz,z\right\rangle _{A}\right\vert
^{2}+\left( 1-\alpha \right)\left( \|S^{\sharp
	_{A}}z\|_A-\frac{\underset{\lambda\in\mathbb C}{\inf}\|S^{\sharp
		_{A}}z-\lambda z\|^2_A}{2\|S^{\sharp
		_{A}}z\|_A}\right)^2\\
	&&\text{(by Lemma \ref{KKK})} \\
&=& \alpha \left\vert \left\langle Sz,z\right\rangle _{A}\right\vert
^{2}+\left( 1-\alpha \right)\|S^{\sharp	_{A}}z\|^2_A\\
&&-( 1-\alpha)\left(\underset{\lambda\in\mathbb C}{\inf}\|S^{\sharp_{A}}z-\lambda z\|^2_A-\frac{\underset{\lambda\in\mathbb C}{\inf}\|S^{\sharp_{A}}z-\lambda z\|^4_A}{4\|S^{\sharp_{A}}z\|^2_A}\right)\\
&\leq&\alpha \left\vert \left\langle Sz,z\right\rangle _{A}\right\vert
^{2}+\left( 1-\alpha \right)\|S^{\sharp	_{A}}z\|^2_A-( 1-\alpha)\eta.
\end{eqnarray*}
So,   
\begin{eqnarray*}
&&\left\vert \left\langle Sz,z\right\rangle _{A}\right\vert ^{2}+\left\Vert
Sz\right\Vert _{A}^{4} \\
&\leq &\alpha \left\vert \left\langle Sz,z\right\rangle _{A}\right\vert
^{2}+\left( 1-\alpha \right) \left\langle SS^{\sharp _{A}}z,z\right\rangle
_{A}+\left\langle S^{\sharp _{A}}Sz,z\right\rangle _{A}^{2}-( 1-\alpha)\eta \\
&\leq &\alpha \left\vert \left\langle Sz,z\right\rangle _{A}\right\vert
^{2}+\left( 1-\alpha \right) \left\langle SS^{\sharp _{A}}z,z\right\rangle
_{A}+\left\langle \left( S^{\sharp _{A}}S\right) ^{2}z,z\right\rangle _{A}-( 1-\alpha)\eta \\
&&\text{(by Lemma \ref{L1.2})} \\
&\leq &\frac{\alpha }{2}\left( \left\Vert Sz\right\Vert _{A}\left\Vert
S^{\sharp _{A}}z\right\Vert _{A}+\left\vert \left\langle Sz,S^{\sharp
_{A}}z\right\rangle _{A}\right\vert \right) +\left( 1-\alpha \right)
\left\langle SS^{\sharp _{A}}z,z\right\rangle _{A}+\left\langle \left(
S^{\sharp _{A}}S\right) ^{2}z,z\right\rangle _{A}\\
&&-( 1-\alpha)\eta \,\,\,\,\,\,\text{(by Lemma \ref{KZ})} \\
&\leq &\frac{\alpha }{4}\left( \left\Vert Sz\right\Vert _{A}^{2}+\left\Vert
S^{\sharp _{A}}z\right\Vert _{A}^{2}\right) +\frac{\alpha }{2}\left\vert
\left\langle S^{2}z,z\right\rangle _{A}\right\vert +\left( 1-\alpha \right)
\left\langle SS^{\sharp _{A}}z,z\right\rangle _{A}+\left\langle \left(
S^{\sharp _{A}}S\right) ^{2}z,z\right\rangle _{A}\\
&&-( 1-\alpha)\eta \,\,\,\,
\text{(by the AM-GM inequality)} \\
&=&\left\langle \left( \frac{\alpha }{4}S^{\sharp _{A}}S+\left( 1-\frac{3 \alpha }{4}
 \right) SS^{\sharp _{A}}+\left( S^{\sharp _{A}}S\right) ^{2}\right)
z,z\right\rangle _{A}+\frac{\alpha }{2}\left\vert \left\langle
S^{2}z,z\right\rangle _{A}\right\vert-( 1-\alpha)\eta\\
&\leq &\left\Vert \frac{\alpha }{4}S^{\sharp _{A}}S+\left( 1-\frac{3 \alpha }{4}
 \right) SS^{\sharp _{A}}+\left( S^{\sharp _{A}}S\right)
^{2}\right\Vert _{A}+\frac{\alpha }{2}\omega _{A}\left( S^2\right)-( 1-\alpha)\eta \text{.}
\end{eqnarray*}
This gives 
\begin{equation*}
d\omega _{A}^{2}\left( S\right) \leq \left\Vert \frac{\alpha }{4}S^{\sharp
_{A}}S+\left( 1-\frac{3  \alpha }{4} \right) SS^{\sharp _{A}}+\left( S^{\sharp
_{A}}S\right) ^{2}\right\Vert _{A}+\frac{\alpha }{2}\omega _{A}\left(
S^2\right)-( 1-\alpha)\eta \text{.}
\end{equation*}
Considering the minimum over $\alpha \in \left[ 0,1\right] $, we get  
\begin{equation}
d\omega _{A}^{2}\left( S\right) \leq \min_{\alpha \in \left[ 0,1\right]
}\left\{ \left\Vert \frac{\alpha }{4}S^{\sharp _{A}}S+\left( 1-\frac{3 \alpha }{4}
 \right) SS^{\sharp _{A}}+\left( S^{\sharp _{A}}S\right)
^{2}\right\Vert _{A}+\frac{\alpha }{2}\omega _{A}\left( S^2\right)-( 1-\alpha)\eta \right\}
\text{.}  \label{BB}
\end{equation}

Similarly, we get  
\begin{equation}
d\omega _{A}^{2}\left( S\right) \leq \min_{\alpha \in \left[ 0,1\right]
}\left\{ \left\Vert \left( 1-\frac{3 \alpha  }{4}\right) S^{\sharp _{A}}S+ 
\frac{\alpha }{4}SS^{\sharp _{A}}+\left( S^{\sharp _{A}}S\right)
^{2}\right\Vert _{A}+\frac{\alpha }{2}\omega _{A}\left( S^2\right)-( 1-\alpha)\mu \right\}. 
  \label{CC}
\end{equation}
Combining (\ref{BB}) and (\ref{CC}) we get the result as desired.
\end{proof}

From Theorem \ref{THH1} (for $\alpha =1$) we obtain the bound:

\begin{cor}
If $S\in {\mathcal{B}}_{A}({\mathcal{H}})$ then  
\begin{align}  \label{eq20}
d\omega _{A}^{2}\left( S \right) \leq \frac{1}{4}\left\Vert S^{\sharp
_{A}}S+SS^{\sharp _{A}}+4\left( S^{\sharp _{A}}S\right) ^{2}\right\Vert _{A}+%
\frac{1}{2}\omega _{A}\left( S^2\right) .
\end{align}
\end{cor}

\begin{remark}
The bound \eqref{eq20} is better than the following bound in \cite{KF_HJM},
namely  
\begin{align}  \label{eq21}
d\omega _{A}^{2}\left( S \right) \leq \frac 1 2 \left[ \omega_{A}\left(
(S^{\sharp _{A}}S)^2+S^{\sharp _{A}}S\right)+\omega_{A}\left((S^{\sharp
_{A}}S)^2-S^{\sharp _{A}}S\right)\right]+\omega_{A}(S^{\sharp _{A}}S^2)
\end{align}
for some operator. If we consider the same example as in Remark \ref{re11},
the bound \eqref{eq21} gives $d\omega _{A}^{2}\left( S\right) \leq 24$
whereas \eqref{eq20} gives $d\omega _{A}^{2}\left( S \right) \leq 20.$
\end{remark}

To achieve the next result we need the inequality.

\begin{lemma}
\cite{H}\label{L11}\bigskip\ \ If $a,c>0$ and $0\leq \alpha \leq 1$ then  
\begin{equation*}
a^{\alpha }c^{1-\alpha } \, \leq \, \alpha a+\left( 1-\alpha \right) c\leq \left( 
\alpha a^{r}+\left( 1-\alpha \right) c^{r}\right) ^{\frac{1}{r}}\text{,} \quad \text{for all $r\geq 1$.}
\end{equation*}

\end{lemma}

\begin{theorem}
\label{th22} If $S\in {\mathcal{B}}_{A}({\mathcal{H}})$ then  
\begin{equation*}
d\omega _{A}^{4n}\left( S\right) \leq 4^{n-1}\left\Vert \left( S^{\sharp
_{A}}S\right) ^{n}+\left( S^{\sharp _{A}}S\right) ^{2n}\right\Vert
_{A}\left\Vert \left( SS^{\sharp _{A}}\right) ^{n}+\left( S^{\sharp
_{A}}S\right) ^{2n}\right\Vert _{A} \quad \text{ for each $n\in 
\mathbb{N}
^{ }$.}
\end{equation*}

\end{theorem}

\begin{proof}
Take $z\in {\mathcal{H}}$ and $\|z\|_A=1,$ we get 
\begin{eqnarray*}
&&\left\vert \left\langle Sz,z\right\rangle _{A}\right\vert ^{2}+\left\Vert
Sz\right\Vert _{A}^{4} \\
&\leq &2\left( \frac{\left\vert \left\langle Sz,z\right\rangle
_{A}\right\vert ^{2n}+\left\Vert Sz\right\Vert _{A}^{4n}}{2}\right) ^{\frac{
1 }{n}} \quad \text{(by Lemma \ref{L11})} \\
&=&2^{1-\frac{1}{n}}\left( \left\vert \left\langle Sz,z\right\rangle
_{A}\right\vert ^{2n}+\left\langle T^{\sharp _{A}}Sz,z\right\rangle
_{A}^{2n}\right) ^{\frac{1}{n}} \\
&\leq &2^{1-\frac{1}{n}}\left( \left\langle S^{\sharp _{A}}Sz,z\right\rangle
_{A}^{\frac{n}{2}}\left\langle SS^{\sharp _{A}}z,z\right\rangle _{A}^{\frac{
n }{2}}+\left\langle S^{\sharp _{A}}Sz,z\right\rangle _{A}^{2n}\right) ^{ 
\frac{ 1}{n}} \quad \text{(by Lemma \ref{LL})} \\
&\leq &2^{1-\frac{1}{n}}\left[ \left( \left\langle S^{\sharp
_{A}}Sz,z\right\rangle _{A}^{n}+\left\langle S^{\sharp
_{A}}Sz,z\right\rangle _{A}^{2n}\right) \left( \left\langle SS^{\sharp
_{A}}z,z\right\rangle _{A}^{n}+\left\langle S^{\sharp _{A}}Sz,z\right\rangle
_{A}^{2n}\right) \right] ^{\frac{1}{2n}} \\
&&\text{(since } (ab+cd )^2\leq {\left( a^{2}+c^{2}\right) \left(
b^{2}+d^{2}\right) }\text{ with }a,b,c,d \geq 0 
\text{)} \\
&\leq &2^{1-\frac{1}{n}}\left( \left\langle \left( S^{\sharp _{A}}S\right)
^{n}z,z\right\rangle _{A}+\left\langle \left( S^{\sharp _{A}}S\right)
^{2n}z,z\right\rangle _{A}\right) ^{\frac{1}{2n}} \\
&&\times \left( \left\langle \left( SS^{\sharp _{A}}\right)
^{n}z,z\right\rangle _{A}+\left\langle \left( S^{\sharp _{A}}S\right)
^{2n}z,z\right\rangle _{A}\right) ^{\frac{1}{2n}} \quad \text{(by Lemma \ref{L1.2})} \\
&\leq &2^{1-\frac{1}{n}}\left( \left\Vert \left( S^{\sharp _{A}}S\right)
^{n}+\left( S^{\sharp _{A}}S\right) ^{2n}\right\Vert _{A}\left\Vert \left(
SS^{\sharp _{A}}\right) ^{n}+\left( S^{\sharp _{A}}S\right) ^{2n}\right\Vert
_{A}\right) ^{\frac{1}{2n}}\text{.}
\end{eqnarray*}

Thus,  
\begin{equation*}
d\omega _{A}^{2}\left( S \right) \leq 2^{1-\frac{1}{n}}\left( \left\Vert
\left( S^{\sharp _{A}}S\right) ^{n}+\left( S^{\sharp _{A}}S\right)
^{2n}\right\Vert _{A}\left\Vert \left( SS^{\sharp _{A}}\right) ^{n}+\left(
S^{\sharp _{A}}S\right) ^{2n}\right\Vert _{A}\right) ^{\frac{1}{2n}}\text{.}
\end{equation*}
This gives the bound.
\end{proof}

From Theorem \ref{th22} (for $n=1$), we obtain the bound:

\begin{cor}
If $S\in {\mathcal{B}}_{A}({\mathcal{H}})$ then 
\begin{align}  \label{eq23}
d\omega_A^2(S)\leq \|SS^{\sharp_A}+(S^{\sharp_A}S)^2\|^{\frac 1
2}_A\|S^{\sharp_A}S+(S^{\sharp_A}S)^2\|^{\frac 1 2}_A.
\end{align}
\end{cor}

\begin{remark}
The bound \eqref{eq23} is better than the existing bound in \cite{BHA},
namely  
\begin{eqnarray}  \label{eq25}
d\omega_A^2(S)\leq
\|S^{\sharp_A}S+(S^{\sharp_A}S)^{\sharp_A}S^{\sharp_A}S\|_A
\end{eqnarray}
for some operators. For the same example as in Remark \ref{re181},  the
bound \eqref{eq25} gives $dw_A^2(S)\leq \frac3 4$, but \eqref{eq23} gives $
dw_A^2(S)\leq \frac {\sqrt{3}}{2\sqrt{2}}.$
\end{remark}

We now obtain a lower bound.

\begin{theorem}
If $S\in {\mathcal{B}}_{A}({\mathcal{H}})$ then  
\begin{equation*}
d\omega _{A}^{2}\left( S\right) \geq \max \left\{ c_{A}^{2}\left( S\right)
\left( 1+\left\Vert S\right\Vert _{A}^{2}\right) ,\omega _{A}^{2}\left(
S\right) \left( 1+c_{A}^{2}\left( S^{\sharp _{A}}S\right) \right) \right\} 
\text{.}
\end{equation*}
\end{theorem}

\begin{proof}
Suppose $z\in {\mathcal{H}}$ and $\|z\|_A=1$. Then  
\begin{eqnarray*}
\left\vert \left\langle Sz,z\right\rangle _{A}\right\vert ^{2}+\left\Vert
Sz\right\Vert _{A}^{4} 
&\geq &\left\vert \left\langle Sz,z\right\rangle _{A}\right\vert
^{2}+\left\vert \left\langle Sz,z\right\rangle _{A}\right\vert
^{2}\left\Vert Sz\right\Vert _{A}^{2} 
\geq c_{A}^{2}\left( S\right) \left( 1+\left\Vert Sz\right\Vert
_{A}^{2}\right) \text{.}
\end{eqnarray*}
This implies 
\begin{equation}
d\omega _{A}^{2}\left( S \right) \geq c_{A}^{2}\left( S\right) \left(
1+\left\Vert S\right\Vert _{A}^{2}\right) \text{.}  \label{A1}
\end{equation}
We also have  
\begin{eqnarray*}
\left\vert \left\langle Sz,z\right\rangle _{A}\right\vert ^{2}+\left\Vert
Sz\right\Vert _{A}^{4} 
&\geq &\left\vert \left\langle Sz,z\right\rangle _{A}\right\vert ^{2}\left(
1+\left\Vert Sz\right\Vert _{A}^{2}\right) 
\geq \left\vert \left\langle Sz,z\right\rangle _{A}\right\vert ^{2}\left(
1+c_{A}^{2}\left( S^{\sharp _{A}}S\right) \right) \text{.}
\end{eqnarray*}
From this we get  
\begin{equation}
d\omega _{A}^{2}\left( S\right) \geq \omega _{A}^{2}\left( S\right) \left(
1+c_{A}^{2}\left( S^{\sharp _{A}}S\right) \right) \text{.}  \label{A2}
\end{equation}
Combining (\ref{A1}) and (\ref{A2}), we get  
\begin{equation*}
d\omega _{A}^{2}\left( S\right) \geq \max \left\{ c_{A}^{2}\left( S\right)
\left( 1+\left\Vert S\right\Vert _{A}^{2}\right) ,\omega _{A}^{2}\left(
S\right) \left( 1+c_{A}^{2}\left( S^{\sharp _{A}}S\right) \right) \right\} 
\text{.}
\end{equation*}
This completes the proof.
\end{proof}

We conclude this section by giving another lower bound.

\begin{theorem}
\label{44} Let $S\in {\mathcal{B}}_{A}({\mathcal{H}})$ and  $m_{A}\left( S\right) =\inf\limits_{\left\Vert x\right\Vert
_{A}=1}\left\Vert Sz\right\Vert _{A}$. Then  
\begin{equation*}
d\omega _{A}^{2}\left( S\right) \geq \max \left\{ \left( 1+m_{A}^{2}\left(
S\right) \right) \omega _{A}^{2}\left( S\right) ,\left( 1+\left\Vert
S\right\Vert _{A}^{2}\right) c_{A}^{2}\left( S\right) \right\} \text{.}
\end{equation*}
\end{theorem}

\begin{proof}
Take $z\in {\mathcal{H}}$ and $\|z\|_A=1$. We obtain  
\begin{eqnarray*}
\left\vert \left\langle Sz,z\right\rangle _{A}\right\vert ^{2}+\left\Vert
Sz\right\Vert _{A}^{4} 
&\geq &\left\vert \left\langle Sz,z\right\rangle _{A}\right\vert
^{2}+\left\vert \left\langle Sz,z\right\rangle _{A}\right\vert
^{2}m_{A}^{2}\left( S\right) 
=\left( 1+m_{A}^{2}\left( S\right) \right) \left\vert \left\langle
Sz,z\right\rangle _{A}\right\vert ^{2}\text{.}
\end{eqnarray*}
This gives  
\begin{equation}
d\omega _{A}^{2}\left( S\right) \geq \left( 1+m_{A}^{2}\left( S\right)
\right) \omega _{A}^{2}\left( S\right) \text{.}  \label{A}
\end{equation}
Moreover, we see  
\begin{eqnarray*}
\left\vert \left\langle Sz,z\right\rangle _{A}\right\vert ^{2}+\left\Vert
Sz\right\Vert _{A}^{4} 
&\geq&\left( 1+\left\Vert Sz\right\Vert _{A}^{2}\right)\left\vert \left\langle
Sz,z\right\rangle _{A}\right\vert ^{2} 
\geq  \left( 1+\left\Vert Sz\right\Vert _{A}^{2}\right)c_{A}^{2}\left(
S\right).
\end{eqnarray*}
So 
\begin{equation}
d\omega _{A}^{2}\left( S\right) \geq \left( 1+\left\Vert S\right\Vert
_{A}^{2}\right) c_{A}^{2}\left( S\right) \text{.}  \label{B}
\end{equation}
Combining (\ref{A}) and (\ref{B}), we get the result as desired.
\end{proof}

It should be mentioned here that the bound in Theorem \ref{44} is also
studied in \cite{BH02} in the setting of Hilbert space operators.

\section{Bounds for the $A_{0}$-Davis-Wielandt radius of block matrices}

\noindent Consider $2\times 2$ block diagonal
matrix 
$
A_{0}\mathbf{}=\left[ 
\begin{array}{cc}
A & 0 \\ 
0 & A%
\end{array}
\right] \text{.}
$
Clearly, $A_{0}\in {\mathcal{B}}\left( {\mathcal{H\oplus H}}\right) ^{}$ is positive.
Therefore, $A_{0}$ induces 
$
\left\langle x,z\right\rangle _{A_{0}}=\left\langle A_{0}x,z\right\rangle
=\left\langle x_{1},z_{1}\right\rangle _{A}+\left\langle
x_{2},z_{2}\right\rangle _{A}\text{,}
$ for all $ x=\left( x_{1},x_{2}\right) $ and $ z=\left( z_{1},z_{2}\right) $ $\in {\mathcal{H\oplus H}}$. If $S_{ij}\in {\mathcal{B}
}_{A}\left( {\mathcal{H}}\right) $ for $i,j=1,2 $,
then $\left[ 
\begin{array}{cc}
S_{11} & S_{12} \\ 
S_{21} & S_{22}
\end{array}
\right] \in {
\mathcal{B}}_{A_{0}}\left( {\mathcal{H\oplus H}}\right) $ and 
$
\left[ 
\begin{array}{cc}
S_{11} & S_{12} \\ 
S_{21} & S_{22}%
\end{array}
\right] ^{^{\sharp _{A_{0}}}}=\left[ 
\begin{array}{cc}
S_{11}^{^{\sharp _{A}}} & S_{21}^{^{\sharp _{A}}} \\ 
S_{12}^{^{\sharp _{A}}} & S_{22}^{^{\sharp _{A}}}
\end{array}
\right] \text{,}$  shown in \cite{Rout}.
To start our analysis we begin with the lemmas.

\begin{lemma}
\cite{BHA}\label{L1} If $ S\in {\mathcal{B}}_{A}({\mathcal{H}})$ and  $
V$ is $A$-unitary, then  
$
d\omega _{A}\left( V^{^{\sharp _{A}}}SV\right) =d\omega _{A}\left( S\right) 
\text{.}
$
\end{lemma}

\begin{lemma}
\cite{Rout}\label{L2} If $S,T\in {\mathcal{B}}_{A}\left( {\mathcal{H}} 
\right) $ then  
$
\omega _{A_{0}}\left( \left[ 
\begin{array}{cc}
S & 0 \\ 
0 & T
\end{array}
\right] \right) =\max \left\{ \omega _{A}\left( S\right) ,\omega _{A}\left(
T\right) \right\} \text{.}
$
\end{lemma}

\begin{lemma}
\cite{Rout}\label{L02} If $S,T\in {\mathcal{B}}_{A}\left( {\mathcal{H}} 
\right) $ then  
$
\omega _{A_{0}}\left( \left[ 
\begin{array}{cc}
S & T \\ 
T & S
\end{array}
\right] \right) =\max \left\{ \omega _{A}\left( S+T\right) ,\omega
_{A}\left( S-T\right) \right\} \text{.}
$
In particular, 
$
\omega _{A_{0}}\left( \left[ 
\begin{array}{cc}
0 & S \\ 
S & 0%
\end{array}
\right] \right) =\omega _{A}\left( S\right) \text{.}
$
\end{lemma}
We now prove our results.

\begin{prop}
\label{PR} If $S,T\in {\mathcal{B}}_{A}\left( {\mathcal{H}}\right) $ then 
$
d\omega _{A_{0}}\left( \left[ 
\begin{array}{cc}
S & 0 \\ 
0 & T
\end{array}
\right] \right) =\max \left\{ d\omega _{A}\left( S\right) ,d\omega
_{A}\left( T\right) \right\} \text{.}
$
\end{prop}

\begin{proof}
Let $X=\left[  
\begin{array}{cc}
T & 0 \\ 
0 & S
\end{array}
\right] $. Suppose $u=\left( z,0\right) \in {\mathcal{H\oplus H}}$, $
\left\Vert z\right\Vert _{A}=1$. Clearly $\left\Vert u\right\Vert _{A_{0}}=1$. Then  $
\left\vert \left\langle Tz,z\right\rangle _{A_{}}\right\vert
^{2}+\left\Vert Tz\right\Vert _{A_{}}^{4}=\left\vert \left\langle
Xu,u\right\rangle _{A_{0}}\right\vert ^{2}+\left\Vert Xu\right\Vert
_{A_{0}}^{4}\leq d\omega _{A_{0}}^{2}\left( X\right) \text{.}
$
Hence, we obtain 
$
d\omega _{A}^{2}\left( T\right) \leq d\omega _{A_{0}}^{2}\left( X\right) 
\text{.}
$
Similarly, we can see that  
$
d\omega _{A}^{2}\left( S\right) \leq d\omega _{A_{0}}^{2}\left( X\right) 
\text{.}
$
Therefore, we infer that  
$
\max \left\{ d\omega _{A}\left( T\right) ,d\omega _{A}\left( S\right)
\right\} \leq d\omega _{A_{0}}\left( X\right) \text{.}
$ 

Now, let $v=\left( z,y\right) \in {\mathcal{H\oplus H}}$ and $
\left\Vert v\right\Vert _{A_{0}}=1$, then  
\begin{eqnarray*}
\sqrt{\left\vert \left\langle Xv,v\right\rangle _{A_{0}}\right\vert
^{2}+\left\Vert Xv\right\Vert _{A_{0}}^{4}} 
&=&\sqrt{\left\vert \left\langle Tz,z\right\rangle _{A}+\left\langle
Sy,y\right\rangle _{A}\right\vert ^{2}+\left\Vert Tz\right\Vert
_{A}^{4}+\left\Vert Sy\right\Vert _{A}^{4}} \\
&\leq &\sqrt{\left\vert \left\langle Tz,z\right\rangle _{A}\right\vert
^{2}+\left\Vert Tz\right\Vert _{A}^{4}}+\sqrt{\left\vert \left\langle
Sy,y\right\rangle _{A}\right\vert ^{2}+\left\Vert Sy\right\Vert _{A}^{4}} \\
&\leq &d\omega _{A}\left( T\right) \left\Vert z\right\Vert _{A}^{2}+d\omega
_{A}\left( S\right) \left\Vert y\right\Vert _{A}^{2} \\
&\leq &\max \left\{ d\omega _{A}\left( T\right) ,d\omega _{A}\left( S\right)
\right\} \text{.}
\end{eqnarray*}
This gives
$
d\omega _{A_{0}}\left( X\right) \leq \max \left\{ d\omega _{A}\left(
T\right) ,d\omega _{A}\left( S\right) \right\} \text{.}
$
\end{proof}

We now discuss the following equality.

\begin{lemma}
\label{L3} If $S,T\in {\mathcal{B}}_{A}\left( {\mathcal{H}}\right) $ then  
$$
d\omega _{A_{0}}\left( \left[ 
\begin{array}{cc}
S & T \\ 
T & S
\end{array}
\right] \right) =d\omega _{A_{0}}\left( \left[ 
\begin{array}{cc}
S-T & 0 \\ 
0 & S+T
\end{array}
\right] \right) \text{ .}
$$
In particular, 
\begin{equation}  \label{L4}
d\omega _{A_{0}}\left( \left[ 
\begin{array}{cc}
0 & S \\ 
S & 0%
\end{array}
\right] \right) =d\omega _{A}\left( S\right) \text{ .}
\end{equation}
\end{lemma}

\begin{proof}
Let $U\mathbf{=}\frac{1}{\sqrt{2}}\left[  
\begin{array}{cc}
I & I \\ 
-I & I%
\end{array}
\right] $ and $W=\left[  
\begin{array}{cc}
S & T \\ 
T & S
\end{array}
\right] $. Then  
$\left[ 
\begin{array}{cc}
P\left( S-T\right) & 0 \\ 
0 & P\left( S+T\right)%
\end{array}
\right]=U^{\sharp _{A_{0}}}WU \text{.}$ 
 Using Lemma \ref{L1}, we get $d\omega _{A_{0}}\left( U^{\sharp{{A_0} }
}WU\right) =d\omega _{A_{0}}\left( W\right) $. So,  
$
d\omega _{A_{0}}\left( \left[ 
\begin{array}{cc}
S & T \\ 
T & S
\end{array}
\right] \right) =d\omega _{A_{0}}\left( \left[ 
\begin{array}{cc}
P\left( S-T\right) & 0 \\ 
0 & P\left( S+T\right)%
\end{array}
\right] \right) \text{.}
$
Taking $U=I$ in Lemma \ref{L1}, we can observe that $d\omega _{A}\left( 
PS\right) =d\omega _{A}\left( S\right) $. So,  
\begin{equation*}
d\omega _{A_{0}}\left( \left[ 
\begin{array}{cc}
S & T \\ 
T & S
\end{array}
\right] \right) =d\omega _{A_{0}}\left( \left[ 
\begin{array}{cc}
P\left( S-T\right) & 0 \\ 
0 & P\left( S+T\right)%
\end{array}
\right] \right) =d\omega _{A_{0}}\left( \left[ 
\begin{array}{cc}
S-T & 0 \\ 
0 & S+T
\end{array}
\right] \right) \text{.}
\end{equation*}
In particular, we get 
\begin{eqnarray*}
d\omega _{A_{0}}\left( \left[ 
\begin{array}{cc}
0 & S \\ 
S & 0%
\end{array}
\right] \right) &=& d\omega _{A_{0}}\left( \left[ 
\begin{array}{cc}
-S & 0 \\ 
0 & S%
\end{array}
\right] \right) \text{ } 
= \max \left\{ d\omega _{A}\left( -S\right) ,d\omega _{A}\left( S\right)
\right\} 
= d\omega _{A}\left( S\right) \text{.}
\end{eqnarray*}
\end{proof}

We now develop a bound of $A_{0}$-Davis--Wielandt radius for $\left[ 
\begin{array}{cc}
0 & B \\ 
C & 0%
\end{array}%
\right] $  in $ {\mathcal{B}}_{A}({\mathcal{H\oplus H}})$.

\begin{theorem}
\label{TT} If $C,B\in {\mathcal{B}}_{A}({\mathcal{H}})$ then  
\begin{eqnarray*}
&&d\omega _{A_{0}}^{2}\left( \left[ 
\begin{array}{cc}
0 & B \\ 
C & 0%
\end{array}
\right] \right) \\
&\leq &\frac{1}{4}\max \left\{ \omega _{A}\left( BB^{\sharp _{A}}+C^{\sharp
_{A}}C+4\left( C^{\sharp _{A}}C\right) ^{2}\right) ,\omega _{A}\left(
B^{\sharp _{A}}B+CC^{\sharp _{A}}+4\left( B^{\sharp _{A}}B\right)
^{2}\right) \right\} \\
&&+\frac{1}{2}\max \left\{ \omega _{A}\left( BC\right) ,\omega _{A}\left(
CB\right) \right\} \text{.}
\end{eqnarray*}
\end{theorem}

\begin{proof}
Let $S=\left[  
\begin{array}{cc}
0 & B \\ 
C & 0%
\end{array}
\right] $ and $z\in {\mathcal{H\oplus H}}$ with $\left\Vert z\right\Vert 
_{A_{0}}=1$. Then  
\begin{eqnarray*}
&&\left\vert \left\langle Sz,z\right\rangle _{A_{0}}\right\vert
^{2}+\left\Vert Sz\right\Vert _{A_{0}}^{4} \\
&\leq &\frac{1}{2}\left( \left\Vert Sz\right\Vert _{A_{0}}\left\Vert
S^{\sharp _{A}}z\right\Vert _{A_{0}}+\left\vert \left\langle Sz,S^{\sharp
_{A}}z\right\rangle _{A_{0}}\right\vert \right) +\left\Vert Sz\right\Vert
_{A_{0}}^{4} \quad \text{(by Lemma \ref{KZ})} \\
&\leq &\frac{1}{4}\left( \left\Vert Sz\right\Vert _{A_{0}}^{2}+\left\Vert
S^{\sharp _{A}}z\right\Vert _{A_{0}}^{2}\right) +\frac{1}{2}\left\vert
\left\langle S^{2}z,z\right\rangle _{A_{0}}\right\vert +\left\Vert
Sz\right\Vert _{A_{0}}^{4} \quad \text{(by AM-GM inequality)} \\
&\leq &\frac{1}{4}\left\langle \left( S^{\sharp _{A}}S+SS^{\sharp
_{A}}\right) z,z\right\rangle _{A_{0}}+\left\langle \left( S^{\sharp
_{A}}S\right) ^{2}z,z\right\rangle _{A_{0}}+\frac{1}{2}\left\vert
\left\langle S^{2}z,z\right\rangle _{A_{0}}\right\vert \quad \text{(by Lemma \ref{L1.2})} \\
&=&\frac{1}{4}\left\langle \left( S^{\sharp _{A}}S+SS^{\sharp _{A}}+4\left(
S^{\sharp _{A}}S\right) ^{2}\right) z,z\right\rangle _{A_{0}}+\frac{1}{2}
\left\vert \left\langle S^{2}z,z\right\rangle _{A_{0}}\right\vert \\
&\leq &\frac{1}{4}\omega _{A_{0}}\left( S^{\sharp _{A}}S+SS^{\sharp
_{A}}+4\left( S^{\sharp _{A}}S\right) ^{2}\right) +\frac{1}{2}\omega
_{A_{0}}\left( S^{2}\right) \\
&=&\frac{1}{4}\omega _{A_{0}}\left( \left[ 
\begin{array}{cc}
BB^{\sharp _{A}}+C^{\sharp _{A}}C+4\left( C^{\sharp _{A}}C\right) ^{2} & 0
\\ 
0 & B^{\sharp _{A}}B+CC^{\sharp _{A}}+4\left( B^{\sharp _{A}}B\right) ^{2}%
\end{array}
\right] \right) \\
&&+\frac{1}{2}\omega _{A_{0}}\left( \left[ 
\begin{array}{cc}
BC & 0 \\ 
0 & CB%
\end{array}
\right] \right) \\
&=&\frac{1}{4}\max \left\{ \omega _{A}\left( BB^{\sharp _{A}}+C^{\sharp
_{A}}C+4\left( C^{\sharp _{A}}C\right) ^{2}\right) ,\omega _{A}\left(
B^{\sharp _{A}}B+CC^{\sharp _{A}}+4\left( B^{\sharp _{A}}B\right)
^{2}\right) \right\} \\
&&+\frac{1}{2}\max \left\{ \omega _{A}\left( BC\right) ,\omega _{A}\left(
CB\right) \right\} \text{} \quad \text{(by Lemma \ref{L2})}.
\end{eqnarray*}
This implies the bound as desired.  

\end{proof}



We now need the lemma.

\begin{lemma}
\cite{KF_HJM}  \label{L} If $x,u,v\in {\mathcal{H}}$ then  
\begin{equation*}
\left\vert
\left\langle x,v\right\rangle _{A}\right\vert ^{2} + \left\vert \left\langle x,u\right\rangle _{A}\right\vert ^{2} \leq   \left\Vert
x\right\Vert _{A}^{2} \left( \max \left\{ \left\Vert v\right\Vert
_{A}^{2},\left\Vert u\right\Vert _{A}^{2}\right\} +\left\vert \left\langle
v,u\right\rangle _{A}\right\vert \right)  \text{.}
\end{equation*}
\end{lemma}

We next obtain another upper bound of $A_{0}$-Davis--Wielandt radius of $\left[ 
\begin{array}{cc}
0 & B \\ 
C & 0%
\end{array}%
\right] $.

\begin{theorem}
\label{th310}  \label{22} If $C,B\in {\mathcal{B}}_{A}({\mathcal{H}})$,
then  
\begin{eqnarray*}
&&d\omega _{A_{0}}^{2}\left( \left[ 
\begin{array}{cc}
0 & B \\ 
C & 0%
\end{array}
\right] \right) \\
&\leq &\frac{1}{2}\max \left\{ \omega _{A}\left( \left( C^{\sharp
_{A}}C\right) ^{2}+C^{\sharp _{A}}C\right) ,\omega _{A}\left( B^{\sharp _{A}}B+ \left(
B^{\sharp _{A}}B\right) ^{2} \right) \right\} \\
&&+\frac{1}{2}\max \left\{ \omega _{A}\left( \left( C^{\sharp _{A}}C\right)
^{2}-C^{\sharp _{A}}C\right), \omega _{A}\left( B^{\sharp _{A}}B-\left( B^{\sharp
_{A}}B\right) ^{2} \right) \right\} \\
&& +\omega _{A_{0}}\left( \left[ 
\begin{array}{cc}
0 & C^{\sharp _{A}}CB \\ 
B^{\sharp _{A}}BC & 0%
\end{array}
\right] \right) \text{.}
\end{eqnarray*}
\end{theorem}

\begin{proof}
Let $S=\left[ 
\begin{array}{cc}
0 & B \\ 
C & 0
\end{array}%
\right]$ and $z\in {\mathcal{H\oplus H}}$ with $\left\Vert z\right\Vert
_{A_{0}}=1$. Replacing $x=z$, $v=Sz$, $u=S^{\sharp _{A}}Sz$ in
Lemma \ref{L}, we obtain 
\begin{eqnarray*}
&&\left\vert \left\langle Sz,z\right\rangle _{A_{0}}\right\vert
^{2}+\left\Vert Sz\right\Vert _{A_{0}}^{4} 
=\left\vert \left\langle z,Sz\right\rangle _{A_{0}}\right\vert
^{2}+\left\langle z,S^{\sharp _{A}}Sz\right\rangle _{A_{0}}^{2} \\
&\leq &\left\Vert z\right\Vert _{A_{0}}^{2}\left( \max \left\{ \left\Vert
Sz\right\Vert _{A_{0}}^{2},\left\Vert S^{\sharp _{A}}Sz\right\Vert
_{A_{0}}^{2}\right\} +\left\vert \left\langle Sz,S^{\sharp
_{A}}Sz\right\rangle _{A_{0}}\right\vert \right)  \\
&=&\frac{1}{2}\left( \left\Vert Sz\right\Vert _{A_{0}}^{2}+\left\Vert
S^{\sharp _{A}}Sz\right\Vert _{A_{0}}^{2}+\left\vert \left\Vert
Sz\right\Vert _{A_{0}}^{2}-\left\Vert S^{\sharp _{A}}Sz\right\Vert
_{A_{0}}^{2}\right\vert \right) +\left\vert \left\langle
S^{2}z,Sz\right\rangle _{A_{0}}\right\vert  \\
&&(\text{since }\max \{ a,b\}  =\frac{1}{2}(a+b+\vert
a-b\vert) \text{ for real scalars } a,b ) \\
&=&\frac{1}{2}\left( \left\langle \left( S^{\sharp _{A}}S+\left( S^{\sharp
_{A}}S\right) ^{2}\right) z,z\right\rangle _{A_{0}}+\left\vert \left\langle
\left( S^{\sharp _{A}}S-\left( S^{\sharp _{A}}S\right) ^{2}\right)
z,z\right\rangle _{A_{0}}\right\vert \right)  
+\left\vert \left\langle S^{\sharp _{A}}S^{2}z,z\right\rangle
_{A_{0}}\right\vert  \\
&\leq &\frac{1}{2}\left[ \omega _{A_{0}}\left( S^{\sharp _{A}}S+\left(
S^{\sharp _{A}}S\right) ^{2}\right) +\omega _{A_{0}}\left( S^{\sharp
_{A}}S-\left( S^{\sharp _{A}}T\right) ^{2}\right) \right]  
+\omega _{A_{0}}\left( S^{\sharp _{A}}S^{2}\right) \text{}\\
&=&\frac{1}{2}\omega _{A_{0}}\left( \left[ 
\begin{array}{cc}
\left( C^{\sharp _{A}}C\right) ^{2}+C^{\sharp _{A}}C & 0 \\ 
0 & B^{\sharp _{A}}B+\left( B^{\sharp _{A}}B\right) ^{2}
\end{array}
\right] \right) \\
&&+\frac{1}{2}\omega _{A_{0}}\left( \left[ 
\begin{array}{cc}
\left( C^{\sharp _{A}}C\right) ^{2}-C^{\sharp _{A}}C & 0 \\ 
0 & \left( B^{\sharp _{A}}B\right) ^{2}-B^{\sharp _{A}}B%
\end{array}
\right] \right) +\omega _{A_{0}}\left( \left[ 
\begin{array}{cc}
0 & C^{\sharp _{A}}CB \\ 
B^{\sharp _{A}}BC & 0%
\end{array}
\right] \right) \\
&=&\frac{1}{2}\max \left\{ \omega _{A}\left( \left( C^{\sharp _{A}}C\right)
^{2}+C^{\sharp _{A}}C\right) ,\omega _{A}\left( B^{\sharp _{A}}B+\left( B^{\sharp
_{A}}B\right) ^{2} \right) \right\} \\
&&+\frac{1}{2}\max \left\{ \omega _{A}\left( \left( C^{\sharp _{A}}C\right)
^{2}-C^{\sharp _{A}}C\right) ,\omega _{A}\left(  B^{\sharp _{A}}B-\left( B^{\sharp
_{A}}B\right) ^{2} \right) \right\} \\
&& +\omega _{A_{0}}\left( \left[ 
\begin{array}{cc}
0 & C^{\sharp _{A}}CB \\ 
B^{\sharp _{A}}BC & 0%
\end{array}
\right] \right) \text{} \quad \text{(by Lemma \ref{L2})}.
\end{eqnarray*}
This implies the inequality as desired.
\end{proof}

\begin{remark}
Considering $B=C$ in Theorem \ref{22}, and using \eqref{L4} and Lemma \ref%
{L02}, we derive the existing inequality (see \cite[Corollary 2.10]%
{KF_HJM})  
\begin{equation*}
d\omega _{A}^{2}\left( B\right) \leq \frac{1}{2}\left[ \omega _{A}\left(
B^{\sharp _{A}}B+\left( B^{\sharp _{A}}B\right) ^{2} \right) +\omega
_{A}\left( B^{\sharp _{A}}B-\left( B^{\sharp _{A}}B\right) ^{2} \right) %
\right] +\omega _{A}\left( B^{\sharp _{A}}B^{2}\right)
\end{equation*}
and so the inequality in Theorem \ref{th310} is a generalization of \cite[%
Corollary 2.10]{KF_HJM}.
\end{remark}

To develop our final result we need the lemma.

\begin{lemma}
\label{lm310}  If $x, u, v\in\mathcal{H}$ then  
\begin{equation*}
|\langle x,u \rangle_A|^2+|\langle x,v \rangle_A|^2\leq \|x\|^2_A\sqrt{
|\langle u,u \rangle_A|^2+2|\langle u,v \rangle_A|^2+|\langle v,v
\rangle_A|^2}.
\end{equation*}
\end{lemma}

\begin{theorem}
\label{th312}  \label{PP}\bigskip\ If $C,B\in {\mathcal{B}}_{A}({\mathcal{H}%
})$ then  
\begin{eqnarray*}
d\omega _{A_{0}}^{4}\left( \left[ 
\begin{array}{cc}
0 & B \\ 
C & 0%
\end{array}
\right] \right) &\leq &\max \left\{ \omega _{A}\left( \left( C^{\sharp
_{A}}C\right) ^{2}+\left( C^{\sharp _{A}}C\right) ^{4}\right) ,\omega
_{A}\left( \left( B^{\sharp _{A}}B\right) ^{2}+\left( B^{\sharp
_{A}}B\right) ^{4}\right) \right\} \\
&&+2\omega _{A_{0}}^{2}\left( \left[ 
\begin{array}{cc}
0 & C^{\sharp _{A}}CB \\ 
B^{\sharp _{A}}BC & 0%
\end{array}
\right] \right) \text{.}
\end{eqnarray*}
\end{theorem}

\begin{proof}
Let $S=\left[  
\begin{array}{cc}
0 & B \\ 
C & 0
\end{array}
\right] $ and $z\in {\mathcal{H\oplus H}}$ with $\left\Vert z\right\Vert 
_{A_{0}}=1$, then
\begin{eqnarray*}
&&\left( \left\vert \left\langle Sz,z\right\rangle _{A_{0}}\right\vert
^{2}+\left\Vert Sz\right\Vert _{A_{0}}^{4}\right) ^{2} 
=\left( \left\vert \left\langle z,Sz\right\rangle _{A_{0}}\right\vert
^{2}+\left\vert \left\langle z,S^{\sharp _{A}}Sz\right\rangle
_{A_{0}}\right\vert ^{2}\right) ^{2} \\
&\leq & \left\Vert z \right\Vert _{A_{0}}^{4}\left( \left\vert \left\langle
Sz,Sz\right\rangle _{A_{0}}\right\vert ^{2}+2\left\vert \left\langle
Sz,S^{\sharp _{A}}Sz\right\rangle _{A_{0}}\right\vert ^{2}+\left\vert
\left\langle S^{\sharp _{A}}Sz,S^{\sharp _{A}}Sz\right\rangle
_{A_{0}}\right\vert ^{2}\right)  
\text{(by Lemma \ref{lm310})} \\
&=&\left\vert \left\langle S^{\sharp _{A}}Sz,z\right\rangle
_{A_{0}}\right\vert ^{2}+2\left\vert \left\langle S^{\sharp
_{A}}S^{2}z,z\right\rangle _{A_{0}}\right\vert ^{2}+\left\vert \left\langle
\left( S^{\sharp _{A}}S\right) ^{2}z,z\right\rangle _{A_{0}}\right\vert ^{2}
\\
&\leq &\left\langle \left( S^{\sharp _{A}}S\right) ^{2}z,z\right\rangle
_{A_{0}}+\left\langle \left( S^{\sharp _{A}}S\right) ^{4}z,z\right\rangle
_{A_{0}}+2\left\vert \left\langle S^{\sharp _{A}}S^{2}z,z\right\rangle
_{A_{0}}\right\vert ^{2} \quad \text{(by Lemma \ref{L1.2})} \\
&\leq &\omega _{A_{0}}\left( \left( S^{\sharp _{A}}S\right) ^{2}+\left(
S^{\sharp _{A}}S\right) ^{4}\right) +2\omega _{A_{0}}^{2}\left( S^{\sharp
_{A}}S^{2}\right) \\
&=&\omega _{A_{0}}\left( \left[ 
\begin{array}{cc}
\left( C^{\sharp _{A}}C\right) ^{2}+\left( C^{\sharp _{A}}C\right) ^{4} & 0
\\ 
0 & \left( B^{\sharp _{A}}B\right) ^{2}+\left( B^{\sharp _{A}}B\right) ^{4}%
\end{array}
\right] \right) 
 +2\omega _{A_{0}}^{2}\left( \left[ 
\begin{array}{cc}
0 & C^{\sharp _{A}}CB \\ 
B^{\sharp _{A}}BC & 0%
\end{array}
\right] \right) \\
&=&\max \left\{ \omega _{A}\left( \left( C^{\sharp _{A}}C\right) ^{2}+\left(
C^{\sharp _{A}}C\right) ^{4}\right) ,\omega _{A}\left( \left( B^{\sharp
_{A}}B\right) ^{4}+\left( B^{\sharp _{A}}B\right) ^{2}\right)\right\} \\
&&+2\omega _{A_{0}}^{2}\left( \left[ 
\begin{array}{cc}
0 & C^{\sharp _{A}}CB \\ 
B^{\sharp _{A}}BC & 0%
\end{array}
\right] \right) \text{} \quad \text{(by Lemma \ref{L2})}.
\end{eqnarray*}
Therefore, this gives the inequality as desired.  

\end{proof}

\begin{remark}
Considering $B=C$ in Theorem \ref{PP}, and using \eqref{L4} and Lemma \ref%
{L02}, we derive the existing inequality (see \cite[Corollary 2.7]%
{KF_HJM}) 
\begin{equation*}
d\omega _{A}^{4}\left( B\right) \leq \omega _{A}\left( \left( B^{\sharp
_{A}}B\right) ^{2}+\left( B^{\sharp _{A}}B\right) ^{4}\right) +2\omega
_{A}^{2}\left( B^{\sharp _{A}}B^{2}\right) 
\end{equation*}%
and so Theorem \ref{th312} is a generalization of \cite[Corollary 2.7]%
{KF_HJM}.
\end{remark}


\noindent \textbf{Declarations. } 
All authors contributed to the preparation of the article.  All authors read and approved the final manuscript.

\noindent \textbf{Competing interests.}  Authors declare that there is no competing interest.

\end{document}